\documentclass[a4paper,10pt]{article}
\usepackage{amsfonts,amsmath,amsthm,dsfont}
\author{Krzysztof Paczka\thanks{Centre of Mathematics for Applications, University of Oslo, Norway, e-mail address: k.j.paczka@cma.uio.no.}}
\title{$G$-martingale representation in the $G$-L\'evy setting\thanks{The research leading to these results has received funding from the European Research Council under the European Community's Seventh Framework Programme (FP7/2007-2013) / ERC grant agreement no [228087].}}
\newtheorem{tw}{Theorem}
\newtheorem{defin}{Definition}
\newtheorem{ass}{Assumption}
\newtheorem{lem}[tw]{Lemma}
\newtheorem{cor}[tw]{Corollary}
\newtheorem{prop}[tw]{Proposition}
\theoremstyle{definition}
\newtheorem{rem}[tw]{Remark}
\DeclareMathOperator*{\esup}{ess\,sup}
\DeclareMathOperator*{\tr}{tr}

\begin{document}
\def\r0{\mathbb{R}^d_0}
    \def\E{\mathbb{E}}
    \def\GE{\hat{\mathbb{E}}}
    \def\cadlag{c\`{a}dl\`{a}g }
        \def\cliprd{C_{b,Lip}(\mathbb{R}^d)}
    \def\cliprn{C_{b,Lip}(\mathbb{R}^n)}
    \def\cliprdn{C_{b,Lip}(\mathbb{R}^{d\times n})}
    \def\cliprm{C_{b,Lip}(\mathbb{R}^m)}
    \def\cp{C^{\infty}_{p}(\mathbb{R}^n)}
    \def\clipr{C_{b,Lip}(\mathbb{R})}
    \def\lipt{Lip(\Omega_t)}
        \def\lipT{Lip(\Omega_T)}
    \def\lip{Lip(\Omega)}
    \def\qB{\langle B\rangle}
    \def\P{\mathbb{P}}
    \def\D{\mathbb{D}}
    \def\I{\mathds{1}}
    \def\N{\mathbb{N}}
        \def\n{\mathcal{N}}
    \def\L12{\mathbb{L}^{1,2}}
    \def\R{\mathbb{R}}
    \def\ae{\mathcal{A}_{\EE}}
    \def\Z{\mathbb{Z}}
        \def\A{\mathcal{A}}
    \def\H{\mathcal{H}}
        \def\tl{\tau_{\lambda}}
     \def\G{\mathcal{G}}
          \def\a{\mathcal{A}}
    \def\C{\mathbb{C}}
        \def\L{\mathbb{L}}
    \def\Q{\mathbb{Q}}
        \def\q{\mathcal{Q}}
    \def\S{\mathcal{S}}
        \def\s{\mathbb{S}}
    \def\B{\mathcal{B}}
        \def\v{\mathcal{V}}
	\def\u{\mathcal{U}}
    \def\b{\mathbb{B}}
    \def\fil{\mathbb{F}}
    \def\F{\mathcal{F}}
    \def\EE{\mathcal{E}}
    \def\m{\mathcal{M}}
    \def\p{\mathcal{P}}
    \def\ito{It\^o }
    \def\levy{L\'evy }
    \def\itolevy{It\^o-L\'evy }
        \def\levykhintchine{L\'evy-Khintchine }
    \def\dom{Dom\,\delta}
    \def\LG{L^2_G(\Omega)}
    \def\LGp{L^p_G(\Omega)}
    \def\LGT{L^2_G(\Omega_T)}
        \def\LGpl{L^2_{G_{\epsilon}}(\mathcal{F}_T)}
    \def\lgT{L^2_G(0,T)}
     \def\mgT{M^2_G(0,T)}
         \def\hSTR{\mathcal{H}^S_G([0,T]\times\mathbb{R}^d_0)}
    \def\hgTR{\mathcal{H}^2_G([0,T]\times\mathbb{R}^d_0)}
        \def\hgTRo{\mathcal{H}^1_G([0,T]\times\mathbb{R}^d_0)}
	\def\hhR{\hat{\mathcal{H}}^2_G([0,T]\times\mathbb{R}^d_0)}
    \def\hgT{\mathcal{H}^2_G(0,T)}
    \def\mg1T{M^1_G(0,T)}
	\def\Mg1T{\mathcal{M}^1_G(0,T)}
    \def\hMT{\hat{\mathcal{M}}^2_G(0,T)}
    \def\hM1T{\hat{\mathcal{M}}^1_G(0,T)}
    \def\hMTp{\hat{\mathcal{M}}^p_G(0,T)}
    \def\mgT{{M}^1_G(0,T)}
    \def\d0{\mathbb{D}_0(\R^+,\R^{d})}
    \def\da{\mathbb{D}_0(\R^+,\R^{2d})}

\maketitle
\begin{abstract}
 In this paper we give the decomposition of a martingale under the sublinear expectation associated with a $G$-\levy process $X$ with finite activity and without drift. We prove that such a martingale consists of an \ito integral w.r.t. continuous part of a $G$-\levy process, compensated \itolevy integral w.r.t. jump measure associated with $X$ and a non-increasing continuous $G$-martingale
starting at 0.
	
	\noindent\textbf{Mathematics Subject Classification 2010:}  60G44, 60G51, 60H05.

\noindent\textbf{Key words:}  $G$-\levy process, \ito calculus, martingale representation, non-linear expectations.
\end{abstract}

\section{Introduction}
In the recent years the problem of model uncertainty and the stochastic calculus under the family of non-dominated probability measures has attracted a lot of attention. The motivation for such problems comes from the finance: the financial models depend on some parameters which are not known a priori and need to be estimated. However, the choice of the parameters may strongly influence the conclusions drawn from the model such as a valuation of derivatives or their hedging strategies. Therefore, one clearly see the necessity of considering a family of models and taking a robust approach to them. 

For many years the mathematicians considered the family of models (i.e. probability measures) which could be dominated by a reference probability measure. Such model uncertainty problems reflect the drift uncertainty and could be analysed using $g$-expectation and BSDE's. However, if the volatility of a financial asset is a source of uncertainty, one needs to consider a family of models which are mutually singular and cannot be dominated by a single reference measure. Shige Peng proposed in \cite{Peng_GBM} to analyse such problems by introducing a process called $G$-Brownian motion  defined on a space equipped in a sublinear expectation called $G$-expectation. Whereas $g$-expectation is defined via BSDE's, $G$-expectation is constructed with viscosity solutions of non-linear heat equation. Denis and Martini  proposed a different formulation of the volatility uncertainty problem via so-called quasi-sure analysis, which directly works with the family of non-dominated probability measures on a canonical space (see \cite{Martini}). It turns out that both approaches are tightly connected (see \cite{Denis_function_spaces}) and lead to the stochastic calculus with the It\^o formula, $G$-SDE's, martingale representation and $G$-BSDE's,  as developed in \cite{Peng_GBM}, \cite{Peng_skrypt}, \cite{Soner_mart_rep}, \cite{Song_Mart_decomp}, \cite{Soner_quasi_anal}, \cite{Peng_general_rep}, \cite{Peng_bsde}, \cite{Girsanov}.

Even though $G$-Brownian motion spurred a lot of interest, its applicability to finance is limited, as most of the financial models rely on the jump processes like \levy processes as a driver of dynamics. However, Peng and Hu  introduced in \cite{Peng_levy} a process called a $G$-\levy process which incorporates three sources of uncertainty: drift, volatility and \levy measure. A $G$-\levy process is a generalization of  $G$-Brownian motion: it is a process consisting of a pure-jump part and a continuous part which might be seen a generalized $G$-Brownian motion (i.e. Brownian motion with both drift and volatility uncertainty). $G$-\levy process is also defined on a space equipped with a sublinear expectation defined by some non-linear IPDE reflecting all three sources of uncertainty. Ren in \cite{Ren} showed that such a sublinear expectation might be represented as supremum of ordinary expectations over a relatively compact family of probability measures (which again cannot be dominated by a single reference probability measure). We also showed in \cite{moj} that these probability measures can be characterized as laws of some \itolevy integrals. For a $G$-\levy process with finite activity we also introduced a good definition of an integral w.r.t. its jump part and we showed that both the \ito formula holds and that (B)SDE's have unique strong solution under the standard Lipschitz conditions. More information on the $G$-\levy processes can be found in Section \ref{sec_preliminaries}.

In this paper we investigate the martingale representation in the $G$-\levy setting. We assume that there is no drift uncertainty and that the volatility and jump uncertainties are unrelated. Under such assumptions we show that a martingale consists of three parts: an \ito integral part w.r.t. $G$-Brownian motion, a non-increasing continuous martingale (which also shows up in the $G$-Brownian motion setting) and a compensated integral w.r.t. Poisson random measure associated with jumps of the $G$-\levy process. The important feature of these three components is that only the $G$-Brownian motion integral is a symmetric martingale (which means that it is a martingale for all considered probability measures), whereas the other components might have only supermartingale property under some probabilities. 

The structure of the paper is as follows. In Section 2 we give an introduction to the framework and present the most important results used throughout the paper. In Section 3 we show that under the second order non-degeneracy condition the viscosity solution of integro-partial DE is smooth. In Section 4 we introduce a compensation of the integral w.r.t. Poisson random measure associated with a $G$-\levy process and show that the compensated integral is a martingale under the sublinear expectation. Section 5 is devoted to the a priori estimates for the postulated decomposition of martingales. The representation for a simple class of random variables is established in Section 6, whereas in Section 7 we show that the decomposition is true for a relatively wide class of random variables. 

\section{Preliminaries}\label{sec_preliminaries}
	Let $\Omega$ be a given space and $\H$  be a vector lattice of real functions defined on $\Omega$, i.e. a linear space containing $1$ such that $X\in\H$ implies $|X|\in\H$. We will treat elements of $\H$ as random variables.
	\begin{defin}\label{def_sublinear_exp}
		\emph{A sublinear expectation} $\E$ is a functional $\E\colon \H\to\R$ satisfying the following properties
		\begin{enumerate}
			\item \textbf{Monotonicity:} If $X,Y\in\H$ and $X\geq Y$ then $\E[ X]\geq\E [Y]$.
			\item \textbf{Constant preserving:} For all $c\in\R$ we have $\E [c]=c$.
			\item \textbf{Sub-additivity:} For all $X,Y\in\H$ we have $\E [X] - \E[Y]\leq\E [X-Y]$.
			\item \textbf{Positive homogeneity:} For all $X\in\H$  we have $\E [\lambda X]=\lambda\E [X]$, $\forall\,\lambda\geq0$.
		\end{enumerate}
		The triple $(\Omega,\H,\E)$ is called \emph{a sublinear expectation space}.
	\end{defin}
	
	We will consider a space $\H$ of random variables having the following property: if\break $X_i\in\H,\ i=1,\ldots n$ then
	\[
		\phi(X_1,\ldots,X_n)\in\H,\quad \forall\ {\phi\in\cliprn},
	\]
	where $\cliprn$ is the space of all bounded Lipschitz continuous functions on $\R^n$. We will express the notions of a distribution and an independence of the random vectors using test functions in $\cliprn$.
	\begin{defin}
		An $m$-dimensional random vector $Y=(Y_1,\ldots,Y_m)$ is said to be independent of an $n$-dimensional random vector $X=(X_1,\ldots,X_n)$ if for every $\phi\in C_{b,Lip}(\R^n\times \R^m)$
		\[
			\E[\phi(X,Y)]=\E[\E[\phi(x,Y)]_{x=X}].
		\]
		Let $X_1$ and $X_2$ be  $n$-dimensional random vectors defined on sublinear random spaces\break $(\Omega_1,\H_1,\E_1)$ and $(\Omega_2,\H_2,\E_2)$ respectively. We say that $X_1$ and $X_2$ are identically distributed and denote it by $X_1 \sim X_2$, if for each $\phi\in\cliprn$ one has
		\[
			\E_1[\phi(X_1)]=\E_2[\phi(X_2)].
		\]
	\end{defin}
	Now we give the definition of $G$-\levy process (after \cite{Peng_levy}).	
	\begin{defin}	
		Let $X=(X_t)_{t\geq0}$ be a $d$-dimensional \cadlag process on a sublinear expectation space $(\Omega,\H, \E)$. We say that $X$ is a \levy process if:
		\begin{enumerate}
			\item $X_0=0$,
			\item for each $t,s\geq 0$ the increment $X_{t+s}-X_{t}$ is independent of $(X_{t_1},\ldots, X_{t_n})$ for every $n\in\N$ and every partition $0\leq t_1\leq\ldots\leq t_n\leq t$,
			\item the distribution of the increment $X_{t+s}-X_t,\ t,s\geq 0$ is stationary, i.e.\ does not depend on $t$.
		\end{enumerate}
		Moreover, we say that a \levy process $X$ is a $G$-\levy process, if satisfies additionally following conditions
		\begin{enumerate}\setcounter{enumi}{3}
			\item there a $2d$-dimensional \levy process $(X^c_t,X^d_t)_{t\geq0}$ such for each $t\geq0$ $X_t=X_t^c+X_t^d$,
			\item\label{condition} processes $X^c$ and $X^d$ satisfy the following growth conditions
			\[
					\lim_{t\downarrow0} \E[|X^c_t|^3]t^{-1}=0;\quad \E[|X^d_t|]<Ct\ \textrm{for all}\ t\geq0.
			\]
		\end{enumerate}
	\end{defin}
	\begin{rem}
		The condition \ref{condition} implies that $X^c$ is a $d$-dimensional generalized $G$-Brownian motion (in particular, it has continuous paths), whereas the jump part $X^d$ is of finite variation.
	\end{rem}
	Peng and Hu noticed in their paper that each $G$-\levy process $X$ might be characterized by a non-local operator $G_X$.
	\begin{tw}[\levykhintchine representation, Theorem 35 in \cite{Peng_levy}]
		Let $X$ be a $G$-\levy process in $\R^d$. For every $f\in C^3_b(\R^d)$ such that $f(0)=0$ we put
		\[
			G_X[f(.)]:=\lim_{\delta\downarrow 0}\, \E[f(X_{\delta})]\delta^{-1}.
		\]
		The above limit exists. Moreover, $G_X$ has the following \levykhintchine representation
		\[
			G_X[f(.)]=\sup_{(v,p,Q)\in \u}\left\{\int_{\r0} f(z)v(dz)+\langle Df(0),q\rangle + \frac{1}{2}\tr[D^2f(0)QQ^T] \right\},
		\]
		where $\r0:=\R^d\setminus\{0\}$, $\u$ is a subset $\u\subset \m(\r0)\times \R^d\times\R^{d\times d}$ and $\m(\r0)$ is a set of all Borel measures on $(\r0,\B(\r0))$. We know additionally that $\u$ has the property
		\begin{equation}\label{eq_property_of_u}
			\sup_{(v,p,Q)\in \u}\left\{\int_{\r0} |z|v(dz)+|q|+ \tr[QQ^T] \right\}<\infty.		
		\end{equation}

	\end{tw}
	\begin{tw}[Theorem 36 in \cite{Peng_levy}]
			Let $X$ be a $d$-dimensional $G$-\levy process.  For each $\phi\in\cliprd$, define $u(t,x):=\E[\phi(x+X_t)]$. Then $u$ is the unique viscosity solution of the following integro-PDE
			\begin{align}\label{eq_integroPDE}
				0=&\partial_tu(t,x)-G_X[u(t,x+.)-u(t,x)]\notag\\
				=&\partial_tu(t,x)-\sup_{(v,p,Q)\in \u}\left\{\int_{\r0} [u(t,x+z)-u(t,x)]v(dz)\right.\notag\\
				&\left.+\langle Du(t,x),q\rangle + \frac{1}{2}\tr[D^2u(t,x)QQ^T] \right\}
			\end{align}
			with initial condition $u(0,x)=\phi(x)$.
	\end{tw}
	
	It turns out that the set $\u$ used to represent the non-local operator $G_X$ fully characterize $X$, namely having $X$ we can define $\u$ satysfying eq.\ (\ref{eq_property_of_u}) and vice versa.
	\begin{tw}
		Let $\u$ satisfy (\ref{eq_property_of_u}). Consider the canonical probability space $\Omega:=\d0$ of all \cadlag functions taking values in $\R^{d}$ equipped with the Skorohod topology. Then there exists a sublinear expectation $\GE$ on $\d0$ such that the canonical process $(X_t)_{t\geq0}$ is a $G$-\levy process satisfying \levykhintchine representation with the same set $\u$.
	\end{tw}
	The proof might be found in \cite{Peng_levy} (Theorem 38 and 40). We will give however the construction of $\GE$, as it is important to understand it. 
	
		Begin with defining the sets of random variables. Put
		\begin{align*}
			\lipT:=&\{\xi\in L^0(\Omega)\colon \xi=\phi(X_{t_1},X_{t_2}-X_{t_1},\ldots,X_{t_n}-X_{t_{n-1}}),\\ &\phi\in C_{b,Lip}(\R^{d\times n}),\ 0\leq t_1<\ldots<t_n<T\},
		\end{align*}
		where $X_t(\omega)=\omega_t$ is the canonical process on the space $\d0$ and $L^0(\Omega)$ is the space of all random variables, which are measurable to the filtration generated by the canonical process. We also set
		\[
			\lip:=\bigcup_{T=1}^{\infty}\, \lipT.
		\]
		Firstly, consider the random variable $\xi=\phi(X_{t+s}-X_{t})$, $\phi\in\cliprd$. We define
		\[
			\GE[\xi]:=u(s,0),
		\]
		where $u$ is a unique viscosity solution of integro-PDE (\ref{eq_integroPDE}) with the initial condition $u(0,x)=\phi(x)$. For general
		\[
			\xi=\phi(X_{t_1},X_{t_2}-X_{t_1},\ldots,X_{t_n}-X_{t_{n-1}}),\quad \phi\in C_{b,Lip}(\R^{d\times n})
		\]
		we set $\GE[\xi]:=\phi_n$, where $\phi_n$ is obtained via the following iterated procedure
		\begin{align*}
			\phi_1(x_1,\ldots,x_{n-1})&=\GE[\phi(x_1,\ldots, X_{t_n}-X_{t_{n-1}})],\\
			\phi_2(x_1,\ldots,x_{n-2})&=\GE[\phi_1(x_1,\ldots, X_{t_{n-1}}-X_{t_{n-2}})],\\
			&\vdots\\
			\phi_{n-1}(x_1)&=\GE[\phi_{n-1}(x_1,X_{t_{2}}-X_{t_{1}})],\\
			\phi_n&=\GE[\phi_{n-1}(X_{t_{1}})].
		\end{align*}
		Lastly, we extend definition of $\GE[.]$ on the completion of $\lipT$ (respectively $\lip$) under the norm $\|.\|_{L^p_G(\Omega)}:=\GE[|.|^p]^{\frac{1}{p}},\ p\geq1$. We denote such a completion by $L^p_G(\Omega_T)$ (or resp. $L^p_G(\Omega)$).
		
		Note that we can equip the Skorohod space $\d0$ with the canonical filtration $\F_t:=\B(\Omega_t)$, where $\Omega_t:=\{\omega_{.\wedge t}\colon \omega\in\Omega\}$. Then using the procedure above we may in fact define the time-consistent conditional sublinear expectation $\GE[\xi|\F_t]$. Namely, w.l.o.g. we may assume that $t=t_i$ for some $i$ and then
		\[
			\GE[\xi|\F_{t_i}]:=\phi_{n-i}(X_{t_{0}},X_{t_1}-X_{t_0},\ldots,X_{t_i}-X_{t_{i-1}}).
		\]
		One can easily prove that such an operator is continuous w.r.t. the norm $\|.\|_1$ and might be extended to the whole space $L^1_G(\Omega)$. By construction above, it is clear that the conditional expectation satisfies the tower property, i.e. is dynamically consistent.
		\begin{defin}
			A stochastic process $(M_t)_{t\in[0,T]}$ is called a $G$-martingale if $M_t\in L^1_G(\Omega_t)$ for every $t\in[0,T]$ and for each $0\leq s\leq t\leq T$ one has
			\[
				M_s=\GE[M_t|\F_s].
			\]
			Moreover, a $G$-martingale $M$ is called symmetric, if $-M$ is also a $G$-martingale.
		\end{defin}

\subsection{Representation of $\GE[.]$ as an upper-expectation}
In \cite{Ren} it has been proven that the sublinear expectation associated with a $G$-\levy process can be represented as an upper-expectation, i.e. supremum of ordinary expectations over some family of probability measures. Moreover, in \cite{moj} we characterized that family of probability measures as laws of some \itolevy integrals under some conditions on the family of \levy measures. (see Section 3 in \cite{moj}). Throughout this paper we will work with the assumption which is stronger than Assumption 1 in \cite{moj}, namely we assume the following.
\begin{ass}\label{ass1}
Let a canonical process $X$ be a $G$-\levy process in $\R^d$ on a sublinear expectation space $(\d0, L^1_G(\Omega), \GE)$. Let $\u\subset \m(\r0)\times \R^d\times\R^{d\times d}$ be a set used in the \levykhintchine representation of $X$ \eqref{eq_integroPDE} satisfying \eqref{eq_property_of_u}. We assume that $\u$ is of the product form $\u=\v\times\{0\}\times\q$, hence the uncertainties connected with the jumps and the volatility are unrelated and there is no drift uncertainty. Moreover, assume that $X$ has finite activity as in Remark 6 in\cite{moj}, i.e.
\[
	\lambda:=\sup_{v\in\v}\, v(\r0)<\infty.
\]
\end{ass}
\begin{rem}
 Let $\G_{\B}$ denote the set of all Borel function $g\colon\R^d\to \R^d$ such that $g(0)=0$. Note that under the finite activity assumption it is easy to construct a measure $\mu\in \m(\R^d)$ such that
\[
	\int_{\r0}|z|\mu(dz)<\infty\quad \textrm{and}\quad \mu(\{0\})=0
\]
and for all $v\in\v$ there exists a function $g_v\in\G_{\B}$ satisfying the following condition
\[
	v(B)=\mu(g_v^{-1}(B))\quad \forall B\in\B(\r0).
\]
One of the possible constructions of such a measure might be obtained via the Knothe-Rosenblatt rearrangement and we can even take $\mu$ as a Lebesgue measure restricted to a hypercube $[0,\lambda^{1/n}]^n$ (see for example \cite{Villiani}, p.8-9).

Consequently, we can re-parametrize set $\u$ in the \levykhintchine formula as follows
	\[
		\tilde \u:=\{(g_v,0,q)\in \G_{\B}\times \R^d\times\R^{d\times d}\colon (v,p,q)\in\u\}.
	\]
\end{rem}
Apart from the finite activity assumption, we will also require that the operator $G_X$ associated with $X$ is non-degenerate in the following sense
\begin{ass}\label{ass_non-degeneracy}
Let $\u=\v\times\{0\}\times\q$ be a set used in the \levykhintchine representation of $X$ \eqref{eq_integroPDE} satisfying \eqref{eq_property_of_u}. We will assume that there exists a finite measure $\pi\in \m(\r0)$ and constants $0<\underline c\leq\bar c<\infty$ s.t. all measures $v\in\v$ and all $B\in\b(\r0)$ one has
\[
	\underline{c}\,\pi(B)\leq v(B) \leq \bar c\,\pi(B).
\]
As a consequence, all $v\in \v$ are equivalent measures to $\pi$ with Radon-Nikodym densities $f_v$ bounded from below and above by $\underline c$ and $\bar c$ on the support of $\pi$.

Let $\s_d$ be a set of all symmetric $d\times d$ matrices. We assume that there exists also a constant $\underline{\sigma}^2>0$ s.t. for all $A,B\in\s_d,\ A\geq B$ one has
\[
	\sup_{Q\in\q}\frac{1}{2}\tr[AQQ^T]-\sup_{Q\in\q}\frac{1}{2}\tr[AQQ^T]\geq \underline{\sigma}^2\tr[A-B].
\]
\end{ass} 
The first part of this assumption is necessary to establish a priori estimates in Section \ref{sec_apriori}, whereas the second part is crucial for having smooth solutions to IPDE \eqref{eq_integroPDE}.

Let $(\tilde \Omega,\G,\P_0)$ be a probability space carrying a Brownian motion $W$ and a \levy process with a \levy triplet $(0,0,\mu)$, which is independent of $W$. Let $N(dt,dz)$ be a Poisson random measure associated with that \levy process. Define
$N_t=\int_{\r0}xN(t,dx)$, which is finite $\P_0$-a.s. as we assume that $\mu$ integrates $|x|$. We also define the filtration generated by $W$ and $N$:
\begin{align*}
	\G_t:=&\sigma\{W_s,\ N_s\colon 0\leq s\leq t\}\vee\n;\ \n:=\{A\in\tilde \Omega\colon \P_0(A)=0\};\ \mathbb G:=(\G_t)_{t\geq0}.
\end{align*}
\begin{tw}[Theorem 7-9 and Corollary 10 in \cite{moj}]\label{tw_rep_sub_lin}
		Introduce a set of integrands $\a_{t,T}^{\u}$, $0 \leq t<T$, associated with $\u$ as a set of all processes $\theta=(\theta^d,0, \theta^{2,c})$ defined on $]t,T]$ satisfying the following properties:
	\begin{enumerate}
		\item $\theta^{c}$ is $\mathbb G$-adapted process and $\theta^d$ is $\mathbb G$-predictable random field on $]t,T]\times \R^d$.
		\item For $\P_0$-a.a. $\omega\in\tilde \Omega$ and a.e. $s\in]t,T]$ we have that 
		$(\theta^d(s,.)(\omega),0, \theta^{c}_s(\omega))\in \tilde\u$.
		\item $\theta$ satisfies the following integrability condition
		\[
		\E^{\P_0}\left[\int_t^T\left[|\theta^{c}_s|^2+\int_{\r0}|\theta^{d}(s,z)|\mu(dz)\right]ds\right]<\infty.
		\]
	\end{enumerate}
	For $\theta\in \a_{0,\infty}^{\u}$ denote the following \levy -\ito integral as
	\[
		B^{t,\theta}_T=\int_{t}^T\, \theta^{c}_sdW_s+ \int_{]t,T]}\int_{\r0}\, \theta^{d}(s,z)N(ds,dz).
	\]
		Lastly, for a fixed $\phi\in \cliprd$ and fixed $T>0$ define for each $(t,x)\in[0,T]\times\R^d$
	\begin{align*}
		u(t,x)&:=\sup_{\theta\in\a_{t,T}^{\u}}\E^{\P_0} [\phi(x+B^{t,\theta}_T)].
	\end{align*}
	Then under Assumption \ref{ass1} $u$ is the viscosity solution of the following integro-PDE
	\begin{equation}\label{eq_IPDE_backwards}
			\partial_tu(t,x)+G_X[u(t,x+.)-u(t,x)]=0
	\end{equation}
	with the terminal condition $u(T,x)=\phi(x)$. Moreover, for every $\xi\in L^1_G(\Omega)$ we can represent the sublinear expectation in the following way	
		\[
		\GE[\xi]=\sup_{\theta\in \a^{\u}_{0,\infty}}\, \E^{\P^{\theta}}[\xi],
	\]
	where  $\P^{\theta}:=\P_0\circ (B{.}^{0,\theta})^{-1},\ \theta\in\a^{\u}_{0,\infty}$. We will introduce also the following notation $\mathfrak{P}:=\{P^{\theta}\colon \theta \in \a^{\u}_{0,\infty}\}$.
\end{tw}
We also have a similar characterization for the conditional sublinear expectation.
\begin{prop}\label{prop_rep_cond_exp}
Under Assumption \ref{ass1} we have for every $\xi\in L^1_G(\Omega)$ that for every $\P\in\mathfrak{P}$
\[
\GE[\xi|\F_t]:=\sideset{}{^{\P}}\esup_{\P'\in\mathfrak{P}(t,\P)}\, \E^{\P'}[\xi|\F_t],\ \P-a.s.
\]
where $\mathfrak{P}(t,P):=\{\P'\in\mathfrak{P}\colon \P=\P'\ \textrm{on}\ \F_t\}$.
\end{prop}
The proof uses exactly the same arguments as in $G$-Brownian motion case (see Proposition 3.4 in \cite{Soner_mart_rep}).
\begin{defin}
We define the capacity $c$ associated with $\GE$ by putting
\[
	c(A):=\sup_{\P\in\mathfrak{P}}\P(A),\quad A\in\B(\Omega).
\]
We will say that a set $A\in\B(\Omega)$ is polar if $c(A)=0$. We say that a property holds quasi-surely (q.s.) if it holds outside a polar set. 
\end{defin}
\begin{rem}
 Note that under Assumption \ref{ass1} the continuous part of a $G$-\levy process $X^c$ is a $G$-Brownian motion (or to be more exact: $G^c$-Brownian motion, we will however drop the distinction as it doesn't lead to any confusion). We will denote that $G$-Brownian notion as $B$ and its quadratic variation is denoted as $\qB$.
 
 Moreover, the finite activity assumption allows us to define the  a Poisson random measure $L(ds,dz)$ associated with the $G$-\levy process $X$ by putting 
 \[
 	L(]s,t],A)=\sum_{s<u\leq t} \I_{A}(\Delta X_u),\ q.s.
 \]
 for any $0<s<t<\infty$ and $A\in\B(\r0)$. The random measure is well-defined and may be used to define the pathwise integral. See \cite{moj}, Section 4 and 5 for details.
\end{rem}

\subsection{Notation for different spaces and integrals}

Throughout this paper we will use different spaces of both random variables and stochastic processes. In this subsection we will define these spaces.
\begin{itemize}
	\item As already mentioned, $L^p_G(\Omega_T),\ p\geq1$ space is a closure of $Lip(\Omega_T)$ under the norm $\|.\|_{L^p_G(\Omega_T)}$.
	\item $\L^p_G(\Omega_T),\ p\geq1$ is the closure of $Lip(\Omega_T)$ under the norm \[
	\|.\|_{\L^p_G(\Omega_T)}:=\GE[\sup_{t\in[0,T]}\GE[|.|\, |\F_t]].\]
	We will also use the following definition: the process $Z$ taking values in a metric space $(\mathcal{X},d)$ is an elementary process, if it has the form
\[
	Z_t=\sum_{n=1}^N\, \phi_n(X_{t_1},\ldots,X_{t_n})\I_{]t_{n-1},t_n]},
\]
where $0\leq t_1<\ldots t_N<\infty$ and $\phi_n\colon\R^{d\times n}\to \mathcal{X}$ is Lipschitz continuous and bounded.

	\item For $ \mathcal{X}=\R,\ \R^d,\ \s_d$ we define the family of elementary processes taking value in $\mathcal X$ as the set of all random processes of the form $	Z_t=\sum_{n=1}^N\, \phi_n(X_{t_1},\ldots,X_{t_n})\I_{]t_{n-1},t_n]}$, 
where $0\leq t_1<\ldots t_N\leq T$ and $\phi_n\colon\R^{d\times n}\to \mathcal{X}$ is Lipschitz continuous and bounded. Let $\H_G^p(0,T; \mathcal{X})$ denote the completion of all $\mathcal X$-valued elementary
	 under the norm
	\[
		\|Z\|_{\H_G^p(0,T; \mathcal{X})}:=\GE\left[\int_0^T|Z_s|^p ds\right]^{1/p}.
	\]
	For a process $Z\in\H_G^p(0,T; \R^d)$ one can define the stochastic integral w.r.t. $G$-Brownian motion $B$ denoted by $\int_0^t Z_s\cdot dB_s$. Similarly for $Z\in\H_G^p(0,T; \R)$ and $\H_G^p(0,T; \R^d)$ one can define  integrals $\int_0^t Z_s ds$ and $\int_0^t Z_s\colon d\qB_s$ (respectively).\footnote{We use the the following notation: $x\cdot y:=x^Ty,\ x,y\in\R^d$ and $A\colon B:=\tr[AB],\ A,B\in\s_d$.} All integrals are continuous operators between spaces $\H_G^p(0,T; \mathcal X)$ and $L^p_G(\Omega)$. See \cite{Peng_skrypt} for details. 
	\item 	Let $\H_G^S([0,T]\times\r0)$ be a space  of all elementary random fields on $[0,T]\times \r0$ of the form
	\begin{equation}\label{eq_K_representation}
			K(u,z)(\omega)=\sum_{k=1}^{n-1}\,\sum_{l=1}^m\, F_{k,l}(\omega)\,\I_{]t_{k},t_{k+1}]}(u)\,\psi_l(z),\ n,m\in\N,
	\end{equation}
	where $0\leq t_1<\ldots<t_n\leq T$ is the partition of $[0,T]$, $\{\psi_l\}_{l=1}^m\subset C_{b,lip}(\R^d)$ are functions with disjoint supports s.t. $\psi_l(0)=0$ and $F_{k,l}=\phi_{k,l}(X_{t_1},\ldots, X_{t_k}-X_{t_{k-1}})$, $\phi_{k,l}\in C_{b,lip}(\R^{d\times k})$.
	We introduce two norms on this space
	\[
		\|K\|_{\H_G^p([0,T]\times\r0)}:=\GE\left[\int_0^T\,\sup_{v\in\v}\,\int_{\r0}|K(u,z)|^pv(dz)du\right]^{1/p},\quad p=1,2
	\]
	and define the completion of $\H_G^S([0,T]\times\r0)$ under these norms as $\H_G^p([0,T]\times\r0)$. For a random field $K\in \H_G^p([0,T]\times\r0)$ we can define a pathwise integral w.r.t. a Poisson jump measure $L(ds,dz)$ as
	\[
		\int_s^t\int_{\r0}K(u,z)L(du,dz):=\sum_{s<u\leq t} K(u,\Delta X_u),\ q.s.
	\]
	The integral is continuous as an operator from space $\H_G^p([0,T]\times\r0)$ to $L^p_G(\Omega)$. For details see Section 5 in \cite{moj}.
\item $\s^p_G(0,T)\, p\geq 1$ is a space of all stochastic process $Z$ such that for each $t\in[0,T]$ $Z_t\in L^p(\Omega_t)$ and $Z$ has finite $\|.\|_{\s^p_G(0,T)}$ norm defined as
\[
	\|Z\|_{\s^p_G(0,T)}:=\GE[\sup_{t\in[0,T]}|Z_t|^p]^{1/p}.
\]
\end{itemize}

\section{Regularity of the solution of integro-PDE}\label{sec_regularity}

In this section we will prove that the integro-partial differential equation \eqref{eq_IPDE_backwards} has not only a viscosity solution, but also a classical solution. We will restrict ourselves to the case, where the terminal condition is smooth and of finite support. Even tough this might be restrictive, it will be sufficient to our purposes. 

We will use in this section the result by Mikulevicius and Pragarauskas about the existence of the classical solution of the class of integro-PDEs on the cylinder $]0,T[\times D$ where $D$ is a bounded domain of $\R^d$. We will extend that result to the the unbounded case by using the estimate of the $C^{2+\alpha}$ norm of the solution and then prove that our equation satisfies all the conditions of the regularity theorem by Mikulevicius and Pragarauskas.

We will introduce now some standard notation used by Mikulevicius and Pragarauskas.  Fix a domain $D$ in $\R^d$ and define a cylinder $Q_s(D):=]s,T[\times D$, $s\in [0,T[$, $Q(D):=Q_0(D)$. Let $\partial'Q:=(]0,T[\times \partial D)\cup ({T}\times D)$. For the multiindex $l=(l_0,l_1,\ldots,l_d)\in \N^{d+1}$ of the parabolic order $|l|=2l_0+l_1+\ldots+l_d$ we shall denote the partial derivative 
\[
	\partial^l u(t,x)=\frac{\partial^{l_0+l_1+\ldots+l_d}}{\partial t^{l_0}\partial x_1^{l_1}\ldots\partial x_d^{l_d}}u(t,x).
\]
of the function $u$ defined on $Q(D)$. The space $C(\bar{Q_s}(D))$ is defined as usual, i.e. space of all continuous functions on $Q_s(D)$ such that the supremum norm is finite. The space $C^{n+\alpha}(\bar{Q_s}(D))$, $n\in\N$, $\alpha\in [0,1[$ is defined as the space of all functions $u$ continuous on $\bar{Q}_s(D)$ together with all their partial derivatives $\partial^l u$, $|l|\leq n$ and having the finite norm 
\[
	\|u\|_{n+\alpha,Q_s(D)}:=\max_{|l|\leq n} \sup_{(t,x)\in[s,T]\times\bar{D}} |\partial^l u(t,x)|+M_{n+\alpha,Q_s(D)},
\]
where
\[
	M_{n+\alpha,Q_s(D)}:=\sup_{t\in]s,T[ }\,M_{n+\alpha}\left[u(t,.);D\right]+\sup_{|l|\leq n,\ x\in D}M_{(n+\alpha-|l|)/2} \left[\partial^l u(.,x); ]s,T[\right]
\]
and
\[
M_{n+\alpha}\left[f;\Omega\right]:= \sup_{|l|=n,\ x,y\in\Omega,\ x\neq y}\frac{|\partial^l f(x)-\partial^l f(y)|}{|x-y|^{\alpha}},\textrm{ where }f\colon\Omega\to\R. 
\]

Introduce now the operator $G$ defining integro-PDE. Let $G=G(A,r,u,w,t,x)$ be a function defined on $\s_{d}\times \R^d\times \R\times C^{2+\alpha}(\bar{Q}(D))\times \bar{Q}(D)$ (where $\s_d$ is a space of all symmetric $d\times d$ matrices) and taking values in $\R$. For brevity let $V:=\s_{d}\times \R^d\times \R$. We fix non-negative constants $\alpha,\nu\in ]0,1[,\ K,\ K_1,\, K_2,\, (K_{\epsilon})_{\epsilon\in[0,1[}$.

\begin{defin}
	Let  $\G (\alpha,\nu,K, K_1, (K_{\epsilon}), D)$ be the class of all functions $G\colon V\times C^{2+\alpha}(\bar{Q}(D))\times \bar{Q}(D)\to \R$ satisfying the following properties:
	\begin{enumerate}
		\item $G$ in convex w.r.t. $A$,
		\item for any $(A,r,u,w,t,x)\in V\times C^{2+\alpha}(\bar{Q}(D))\times \bar{Q}(D)$ and $\xi\in\R^d$ we have
		\[
			\nu|\xi|^2\leq G(A+\xi\xi^T,r,u,w,t,x)-G(A,r,u,w,t,x)\leq \nu^{-1}|\xi|^2,
		\]
		\item for any $A,A'\in \s_d$, $r,r'\in\R^d$, $u,u'\in\R$, $w\in C^{2+\alpha}(\bar{Q}(D))$ and $(t,x,)\in \bar{Q}(D)$ we have
		\[
			|G(A,r,u,w,t,x)-G(A',r',u',w,t,x)|\leq K \left[\sum_{i,j=1}^d |A_{ij}-A'_{ij}|+\sum_{i=1}^d |r_i-r'_i|+|u-u'|\right]
		\]
		\item for any $\epsilon\in]0,1[,\ t\in[0,T[,\ (v,w)\in V\times C^{2+\alpha}(\bar{Q}(D))$ we have
		\[
			\|G(v,w,.)\|_{\alpha,Q_t(D)}\leq K|v|+\epsilon \|w\|_{2+\alpha,Q_t(D)}+K_{\epsilon}\|w\|_{0,Q_t(D)}+K_1,
		\]
		\item if $w,w_n\in C^{2+\alpha}(\bar{Q}(D))$ , $\sup_{n} \|w_n\|_{2+\alpha,Q(D)}<\infty$, $w_n\to w$ in $C^2(\bar{Q}(D))$ as $n\to\infty$, then $G(v,w_n,t,x)\to G(v,w,t,x)$ for all $(v,t,x) \in V\times Q(D)$. 
	\end{enumerate}
	
	Moreover, we say that $G$ belongs to the class  $\G_1 (\alpha,\nu,K, K_1,K_2, (K_{\epsilon}),D)$, if $G$ belongs to the class $\G (\alpha,\nu,K, K_1, (K_{\epsilon}))$ and is infinitely differentiable w.r.t. $v$ and for any $v\in V$, $w,w'\in C^{2+\alpha}(\bar{Q}(D))$ and $t\in[0,T[$ one has
	\[
		|G(v,w,.)-G(v,w',.)|_{0,Q_t(D)}\leq K_2 |w-w'|_{0,Q_t(D)}.
	\]
\end{defin}

\begin{ass}\label{assumptions_iPDE}
	Fix $\alpha\in ]0,1[$, $\delta>0$, $K>0$. Let $D$ be bounded non-empty domain $D=\{x\in\R^d\colon \psi(x)>0\}$ with boundary $\partial D=\{\psi=0\}$, where the function $\psi\in C^{2+\alpha}(\R^d)$ is such 
	\[
		\|\psi\|_{C^{2+\alpha}(\R^d)}\leq K \quad \textrm{and}\quad \sup_{x\in\partial D} |\nabla \psi|\geq\delta.
	\]
	Let $G=G(A,r,u,w,t,x)$ be a function defined on $\s_{d}\times \R^d\times \R\times C^{2+\alpha}(\bar{Q}(D))\times \bar{Q}(D)$  and taking values in $\R$ such that $G(0,0,0,0,T,.)=0$ on $\partial D$. 
	
	Moreover, assume that $G$ belongs to the class $\G (\alpha,\nu,K, K_1, (K_{\epsilon}),D)$ for some non-negative constants $\nu\in ]0,1[,\ K,\ K_1,\, (K_{\epsilon})_{\epsilon\in[0,1[}$ and there exists a sequence $\{G_n\}_n$ such that $G_n$ are in the class $\G_1 (\alpha,\nu,K, K_1,K^n_2, (K_{\epsilon}),D)$ for some non-negative constant $K_2^n$, $n\in\N$, and $G_n(.,t,x)\to G(.,t,x)$ uniformly on bounded sets of $V\times C^{2+\alpha}(\bar{Q}(D))$ for any $(t,x)\in Q(D)$.
\end{ass}

\begin{tw}[Theorem 1 in \cite{classical_bound}]\label{tw_existence_classical}
	Under Assumption \ref{assumptions_iPDE} consider the problem
	\[
		\partial_t u(t,x)+G(D^2u(t,x),Du(t,x),u(t,x),u,t,x)=0\quad \textrm{in}\ Q(D),\quad u=0\quad \textrm{on}\ \partial'Q(D).
	\]
	Then this problem has the classical solution $u\in C^{2}(\bar{Q}(D))$ such that 
	\[
		\|u\|_{2,Q(D)}\leq NK_1
	\]
	and the constant $N$ depends only on $\alpha,\ d,\ \nu,\ \delta,\ K,\ (K_{\epsilon})$.
\end{tw}

We stress that the bound depends on the domain $D$ only via parameters $\delta$ and $K$. Moreover, for any $r>0$ the open ball $D_r\subset \R^d$ centred at $0$ and with radius $r$ satisfies the assumptions in Assumption \ref{assumptions_iPDE} for any $\delta>0$ and $K>0$. We will use this fact to extend the existence result to the case $D=\R^d$ via solving the equation on domains $D_n$ and taking a convergent subsequence. Thus let us introduce the definition
\begin{defin}\label{def_class_tilde_G}
	We say that function $G\colon\s_{d}\times \R^d\times \R\times C^2(\bar{Q}(\R^d))\times \bar{Q}(\R^d)\to\R$ belongs to the class $\tilde{\G} (\alpha,\nu,K, K_1, (K_{\epsilon}),\R^d)$ if the following conditions are satisfied:
	\begin{enumerate}
		\item for each open ball $D_n$ with radius $n$ centred at $0$ there exists a function $G_n$ such that  
	$G_n\in \G (\alpha,\nu,K, K_1, (K_{\epsilon}),D_n)$ and $G_n$ satisfies the following coordination condition: $G_n(0,0,0,0,T,.)=0$ on $\partial D_n$;
		\item for any $A_n, A\in\s_d$, $r_n,r\in \R^d$ $u_n,u\in\R$ and $w_n\in C^{2+\alpha}(\bar Q(D_n))$, $w\in C^2(\bar{Q}(\R^d))$ such that 
		\begin{itemize}
		\item $A_n\to A$, $r_n\to r$, $u_n\to u$ as $n \to\infty$, 
		\item $\sup_{n\geq m} \|w_n|_{\bar{Q}(D_m)}\|_{2+\alpha,Q(D_m)}<\infty$ for all $m\in\N$, 
		\item for each $m\geq1$ we have  $w_n|_{\bar Q(D_m)}\to w|_{\bar Q(D_m)}$ in $C^2(\bar Q(D_m))$ as $n\to\infty,\ n\geq m$,
		\end{itemize} 
		then $G_n(A_n,r_n,w_n,t,x)\to G(A,r,w,t,x)$ as $n\to\infty$ eventually for all $(t,x) \in Q(\R^d)$. 
	\end{enumerate}	
	Moreover we say  that $G$ belongs to the class  $\tilde{\G}_1 (\alpha,\nu,K, K_1,(K_2^n), (K_{\epsilon}),\R^d)$, if $G$ belongs to the class $\tilde{\G} (\alpha,\nu,K, K_1, (K_{\epsilon}),\R^d)$, each $G_n$ from point 1 belongs to ${\G}_1 (\alpha,\nu,K, K_1,K_2^n, (K_{\epsilon}),D_n)$. 
\end{defin}

\begin{cor}\label{cor_existence_classical_Rd}
	Let $G$ be in the class $\tilde{\G}_1 (\alpha,\nu,K, K_1, (K_2^n), (K_{\epsilon}),\R^d)$
	for some non-negative constants $\nu\in ]0,1[,\ K,\ K_1,\, (K_{\epsilon})_{\epsilon\in[0,1[}$ and $(K_2^n)_{n\geq1}$.
	
	Then the problem:
	\[
		\partial_t u(t,x)+G(D^2u(t,x),Du(t,x),u(t,x),u,t,x)=0\quad \textrm{in}\ Q(\R^d),\quad u(T,.)=0.
	\]
	has the classical solution $u\in C^{2}(\bar{Q}(\R^d))$ such that 
	\[
		\|u\|_{2,Q(\R^d)}\leq NK_1
	\]
		and the constant $N$ depends only on $\alpha,\ d,\ \nu,\ \delta,\ K,\ (K_{\epsilon})$. 
\end{cor}

\begin{proof}
By the assumptions on $G$ and Theorem \ref{tw_existence_classical} we know that for each $n\in\N$ there exists a classical solution $u^n$ for the problem
	\[
		\partial_t u^n(t,x)+G_n(D^2u^n(t,x),Du^n(t,x),u^n(t,x),u^n,t,x)=0\quad \textrm{in}\ Q(D_n),\quad u=0\quad \textrm{on}\ \partial'Q(D_n).
	\]
	such that 
	\[
		\|u^n\|_{2+\alpha,Q(D_n)}\leq NK_1.
	\]
	Let $n\geq m$, where $m\geq 1$ is fixed. Then of course we have also following bound
	\[
		\|u^n|_{\bar{Q}(D_m)}\|_{2+\alpha,Q(D_m)}\leq NK_1.
	\]
	Note that the set $\bar{Q}(D_m)$ is compact for every $m\in\N$ thus the family $\{u^n|_{\bar{Q}(D_m)}\colon n\geq m\}$ is relatively compact in the $C^{2}(\bar{Q}(D_m))$ topology (by the Arzel\`a-Ascoli theorem). Thus we can choose the sequence $(n_k)_k$ by using the diagonal argument, such that 
	$\{u^{n_k}|_{\bar{Q}(D_m)}\colon k\geq m\}$ is a Cauchy sequence in $C^2({\bar{Q}(D_m)})$ for each $m\in \N$. Thus we may define a unique function $u\in C^2_{loc}({\bar{Q}(\R^d)})$ as the limit of these Cauchy sequences. 
	
	We claim that $u$ belongs to $C^2({\bar{Q}(\R^d)})$ and that $u$ is the solution of the integro-PDE. The first assertion follow from the fact for each $n\geq m$ we have a bound
\[
		\|u^n|_{\bar{Q}(D_m)}\|_{2,Q(D_m)}\leq \|u^n|_{\bar{Q}(D_m)}\|_{2+\alpha,Q(D_m)}\leq NK_1.
\]
thus by the definition of $u$ we easily get
\[
		\|u|_{\bar{Q}(D_m)}\|_{2,Q(D_m)}\leq NK_1 \quad\textrm{and}\quad \|u\|_{2,Q(\R^d)}\leq NK_1.
\]
Now we can easily prove the second assertion by noting that the sequence 
\[
(D^2u^{n_k}(t,x), Du^{n_k}(t,x), u^{n_k}(t,x),u^{n_k})_{k\in\N}
\] 
satisfies the assumptions of Definition \ref{def_class_tilde_G}, point 2, so for each $(t,x)\in Q(\R^d)$  we have that $G_{n_k}(D^2u^{n_k}(t,x), Du^{n_k}(t,x), u^{n_k}(t,x),u^{n_k},t,x)\to G(D^2u(t,x), Du(t,x), u(t,x),u,t,x)$ as \break $k\to\infty$ (and of course $\partial_t u^{n_k}(t,x)\to \partial_t u(t,x)$ as $k\to\infty$). Thus $u$ solves our equation.
\end{proof}

Using this corollary we are able to prove that the integro-PDE  \eqref{eq_IPDE_backwards}  has a classical solution if the terminal condition is sufficiently regular. 
\begin{prop}\label{prop_existence_classical_phi}
	Let $\phi\in C^2(\R^d)$. Then under Assumption \ref{ass_non-degeneracy} the equation  \eqref{eq_IPDE_backwards} has a classical solution.
\end{prop}
\begin{proof}
	Note that we may rewrite the equation by introducing the new operator
	\[
		G(A,w,t,x):=G^c(A)+G^d(w,t,x):=\sup_{Q\in\q}\frac{1}{2} \tr\left[AQQ^T\right]+\sup_{v\in\v}\int_{\r0}\left[w(t,x+z)-w(t,x)\right]v(dz)
	\]
	Define $G_{\phi}(A,r,u,w,t,x):=G(A+D^2\phi(x),w+\phi(x),t,x):=G^c_{\phi}(A)+G^d_{\phi}(w,t,x)$.
	
	It is easy to notice that $v$ is the classical solution of the equation \eqref{eq_IPDE_backwards}  if and only if $u(t,x):=v(t,x)-\phi(x)$ is the classical solution of  the following equation
	\begin{equation}\label{eq_iPDE_phi}
		\partial_t\, u(t,x)+G_{\phi}(D^2u(t,x),u,t,x)=0,\quad u(T,.)\equiv 0.
	\end{equation}
	It is rather clear that due to the non-degeneracy assumption, $G_{\phi}$ satisfies the conditions in Corollary \ref{cor_existence_classical_Rd} for some constants which depend only on the set $\u$ (or rather $\q$), $d$ and $\|\phi\|_{2,\R^d}$. Moreover, by regularizing $G^c$ using the smooth approximation of unity, we can get the existence of the sequence of operators, which is smooth in the first variable. Hence, using the corollary we get the existence of the solution $u$ and thus also $v$. 
\end{proof}

\section{Compensated pure-jump processes}
In this section we will consider 'compensated' $G$-\itolevy integral and prove that such a process is a $G$-martingale. Such a result is a direct analouge of the fact that the integral w.r.t. compensated Poisson random measure is a martingale. In our case however we don't know how to compensate the jump measure, thus we will need to compensate the whole integral.

To be more exact, for a pure-jump integral $\int_0^{t}\int_{\r0}K(s,z)L(ds,dz),\ K\in\hgTR$ we consider the compensated integral defined as
\[
	\int_0^{t}\int_{\r0}K(s,z)L(ds,dz)-\int_0^{t}\sup_{v\in\v}\int_{\r0}K(s,z) v(dz)ds.
\]
The first thing we will show is that the correction term lies in a appropriate space.

\begin{prop}\label{prop_compensated_in_L_G}
	Under Assumption \ref{ass1} and \ref{ass_non-degeneracy} for each $K\in\hgTR$ and $t\in[0,T]$
	\[
		J_t(K):=\int_0^{t}\sup_{v\in\v}\int_{\r0}K(s,z) v(dz)ds
	\]
	is an element of $L^2_G(\Omega_t)$.
\end{prop}

\begin{proof}
	First, we will prove that the assertion is true for $K\in \hSTR$. In fact by the linearity of the integral w.r.t. time and by the fact that $L^2_G(\Omega_t)$ is a linear space, we can to consider $K$ of the following form
	\[
		K(s,z):=\,\I_{]t_1,t_2]}(s)\, \sum_{k=1}^mF_k\psi_k(z),
	\]
	where $0\leq t_1<t_2\leq T$, $F_k\in Lip(\Omega_{t_1})$, $\psi_k \in C_{b,lip}(\R^d)$ such that $\psi_k(0)=0$. Assume additionally that $F_k$ has the following representation:
	\[
		F_k=\phi_k(X_{s_1},X_{s_2}-X_{s_1},\ldots,X_{s_{n-1}}-X_{s_n}),\ \phi_k\in C_{b,lip}(\R^{d\times n}),\ 0\leq s_1<\ldots< s_n\leq t_1.
	\]
	For such a simple $K$ we have
	\[
	J_t(K)=(t_2\wedge t -t_1\wedge t)\, \sup_{v\in\v}\,\sum_{k=1}^m F_k\int_{\r0}  \psi_k(z)v(dz).
	\]	
	We will prove that $J_t(K)\in Lip(\Omega_t)$. Consider the function
	\[
		 \phi(x_1,\ldots,x_n):=\sup_{v\in\v}\,\sum_{k=1}^m \phi_k(x_1,\ldots,x_n)\int_{\r0}  \psi_k(z)v(dz).
	\]
	$\psi$ is bounded because $\phi_k$ and $\psi_k$ are bounded. We will prove now that $\phi$ is Lipschitz continuous. Let $x,y\in\R^{d\times n}$
	\begin{align*}
		|\phi(x)-\phi(y)|=&\left|\sup_{v\in\v}\,\sum_{k=1}^m \phi_k(x)\int_{\r0}  \psi_k(z)v(dz)-\sup_{v\in\v}\,\sum_{k=1}^m \phi_k(y)\int_{\r0}  \psi_k(z)v(dz)\right|\\
		\leq& \sup_{v\in\v}\left|\,\sum_{k=1}^m \phi_k(x)\int_{\r0}  \psi_k(z)v(dz)-\sum_{k=1}^m \phi_k(y)\int_{\r0}  \psi_k(z)v(dz)\right|\\
		\leq& \sup_{v\in\v}\,\sum_{k=1}^m \left|\phi_k(x)-\phi_k(y)\right|\int_{\r0}  |\psi_k(z)|v(dz)\\
		\leq& \sum_{k=1}^m \left|\phi_k(x)-\phi_k(y)\right|\,\sup_{v\in\v}\int_{\r0}  |\psi_k(z)|v(dz)\leq L|x-y|
	\end{align*}
	for some constant $L>0$ as all $\phi_k$ are Lipschitz continuous and all $\psi_k$ are bounded.	As the conclusion we get that $J_t(K)\in Lip(\Omega_t)$.
	
	Finally, we notice that $J_t$ is a Lipschitz-continuous function from $\hSTR$ to $L^2_G(\Omega_t)$ thus we may easily get the assertion of the theorem for all $K\in \hgTR$.
\end{proof}

Now we can state the main result of this section.
\begin{tw}\label{tw_compensated_jump_mart}
	Assume Assumption \ref{ass1} and \ref{ass_non-degeneracy}. For a fixed $K\in\hgTRo$ define the compensated pure-jump integral
\[
	M_t:=\int_0^{t}\int_{\r0}K(s,z)X(ds,dz)-\int_0^{t}\sup_{v\in\v}\int_{\R^d}K(s,z) v(dz)ds.
\]
	Then for any non-increasing $G$-martingale $N_t=\int_0^t H_s\colon\qB_s-\int_0^T\sup_{Q\in\q}\tr[H_sQQ^T]ds$, $\mathcal H^1_G(0,T;\s_d)$ we have that $M+N$ is a $G$-martingale. In particular, taking $N\equiv 0$ we get that $M$ is also a $G$-martingale.
\end{tw}

\begin{proof}
	It is sufficient to consider the processes $K$ and $H$ of the form
	\[
			K(s,z)=\sum_{k=1}^{n-1}\,\sum_{l=1}^m\, F_{k,l}\,\I_{]t_{k},t_{k+1}]}(s)\,\psi_l(z),\quad H_s=\sum_{k=1}^{n-1}\, G_{k}\,\I_{]t_{k},t_{k+1}]}(s)
	\]
	where $0\leq t_1<\ldots<t_n\leq T$ is the partition of $[0,T]$, $\{\psi_l\}_{l=1}^m\subset C_{b,lip}(\R^d)$ are functions with disjoint supports s.t. $\psi_l(0)=0$ and $F_{k,l}=\phi_{k,l}(X_{t_1},\ldots, X_{t_k}-X_{t_{k-1}})$, $G_{k}=\phi_{k}(X_{t_1},\ldots, X_{t_k}-X_{t_{k-1}})$ with $\phi_{k,l}\in C_{b,lip}(\R^{d\times k})$ and $ \phi_k\in C_{b,lip}(\R^{d\times k};\s_d)$ (i.e. it is a bounded Lipschitz function of $\R^{d\times k}$ taking values in $\s_d$). In fact we need only to consider the one-step case
	\begin{align*}
		\GE[M_{t_{k+1}}&-M_{t_{k}}+N_{t_{k+1}}-N_{t_{k}}|\F_{t_k}]\\=& \GE\left[\sum_{l=1}^m F_{k,l}\sum_{t_k<s\leq t_{k+1}} \psi_l(\Delta X_s) +\tr[G_k(\qB_{t_{k+1}}-\qB_{t_{k}})\Big|\F_{t_k}\right]\\
		&-\left[\sup_{v\in\v}\sum_{l=1}^m F_{k,l}\sum_{t_k<s\leq t_{k+1}}\int_{\R^d} \psi_l(z)v(dz)+\sup_{Q\in\q}\tr[G_kQQ^T]\right](t_{k+1}-t_{k})=:A-B
	\end{align*}
	We want to prove that $A=B$. Let $\Delta X:=(X_{t_1},\ldots,X_{t_k}-X_{t_{k-1}})$. By the definition of the conditional expectation it is easy to see that
	\[
		A=\GE\left[\sum_{l=1}^m \phi_{k,l}(x)\sum_{t_k<s\leq t_{k+1}} \psi_l(\Delta X_s) +\tr[\phi_k(x)(\qB_{t_{k+1}}-\qB_{t_{k}})\right]\Big|_{x=\Delta X}.
	\]
	Now we use Theorem \ref{tw_rep_sub_lin} to transform the sublinear expectation $\GE[.]$ into an upper-expectation using the argument just as in Theorem 20 in \cite{moj}:
	\begin{align*}
			A&=\sup_{\theta\in\a^{\u}_{0,T}}\E^{\P_0}\left[\sum_{l=1}^m \phi_{k,l}(x)\sum_{t_k<s\leq t_{k+1}} \psi_l(\theta^d(s,\Delta N_u) ) +\tr[\phi_k(x)\int^{t_{k+1}}_{t_{k}}\theta_s^c(\theta^c_s)^Tds]\right]\Big|_{x=\Delta X}\\
			&=\sup_{\theta\in\a^{\u}_{0,T}}\E^{\P_0}\left[\sum_{l=1}^m \phi_{k,l}(x)\int_{t_k}^{ t_{k+1}}\int_{\r0} \psi_l(\theta^d(s,z) )N(ds,dz) +\tr[\phi_k(x)\int^{t_{k+1}}_{t_{k}}\theta_s^c(\theta^c_s)^Tds]\right]\Big|_{x=\Delta X}\\
			&=\sup_{\theta\in\a^{\u}_{0,T}}\E^{\P_0}\left[\sum_{l=1}^m \phi_{k,l}(x)\int_{t_k}^{ t_{k+1}}\int_{\r0} \psi_l(\theta^d(s,z) )\mu(dz)ds +\tr[\phi_k(x)\int^{t_{k+1}}_{t_{k}}\theta_s^c(\theta^c_s)^Tds]\right]\Big|_{x=\Delta X}\\
&=\left[\sup_{v\in\v}\sum_{l=1}^m \phi_{k,l}(x)\int_{\R^d} \psi_l(z)v(dz) +\sup_{Q\in\q}\tr[\phi_k(x)QQ^T]\right](t_{k+1}-t_k)\Big|_{x=\Delta X}=B.\qedhere
		\end{align*}
\end{proof}
Note that the compensated pure-jump integral is a $G$-martingale, but it is not symmetric in general under the Assumption \ref{ass1}. Thus it has a nature which is completely different from the \ito integral w.r.t. $G$-Brownian motion.

\section{A priori estimates for the $G$-martingale decomposition}\label{sec_apriori}

In this section we will assume that a $G$-martingale $M$ has the following decomposition
\begin{equation}\label{eq_representation_M}
	M_t = M_0 + \int_0^t H_s \cdot dB_s -K^c_t +\int_0^t\int_{\R^d} K^d(s,z)L(ds,dz)-\int_0^t \sup_{v\in\v} \int_{\r0} K^d(s,z)v(dz)ds\ q.s.,
\end{equation}
where $H\in \mathcal{H}^2_G(0,T;\R^d)$, $K^c$ is a non-decreasing process in $\s^2_G(0,T)$  such that $-K^c$ is a $G$-martingale and $K^c_0=0$ and $K^d \in \hgTR$. We will give the estimates of the norms of these processes in terms of the process $M$. 

\begin{tw}\label{tw_estimates1}
	Let $M$ has the decomposition as in eq. (\ref{eq_representation_M}). Under Assumption \ref{ass1} and \ref{ass_non-degeneracy} there exists a constant $C$ depending only on the dimension $d$ such that 
	\[
		\|H\|^2_{ \mathcal{H}^2_G(0,T;\R^d)}+\|K^c\|^2_{\s^2_G(0,T)}+\|K^d\|^2_{\hgTR}\leq C \|M\|^2_{\s^2_G(0,T)}.
	\]
\end{tw}
\begin{proof}
	We will follow the idea in \cite{Soner_mart_rep}. Applying the \ito formula for a \itolevy process $M$ we easily get
	\begin{align}\label{eq_M2_dynamics}
			M_T^2&=M_t^2+2\int_t^T M_{s-}H_s\cdot dB_s+\int_t^T H_sH_s^T\colon d\qB_s-2\int_t^TM_{s-}dK^c_t\notag\\
			&-2\int_t^TM_{s-}\sup_{v\in\v}\int_{\r0}K^d(s,z)v(dz)ds+\int_t^T\int_{\r0}\left[(M_{s-}+K^d(s,z))^2-(M_{s-})^2\right]L(ds,dz)
	\end{align}
	Note that the last term might be rewritten as $\int_t^T\int_{\r0}\left[2M_{s-}K^d(s,z) +(K^d(s,z))^2\right]L(ds,dz)$.

	Fix $\P\in\p$. Note that by exactly the same argument as in the proof of Theorem 20 in \cite{moj} we have that for every $K\in\hgTRo$ we have 
	\begin{multline}\label{eq_norm_estimate_integral}
			\E^{\P}\left[\int_t^T\inf_{v\in\v}\int_{\r0}K(s,z)v(dz)ds\right]\leq \E^{\P}\left[\int_t^T\int_{\r0}K(s,z)L(ds,dz)\right]\\\leq \E^{\P}\left[\int_t^T\sup_{v\in\v}\int_{\r0}K(s,z)v(dz)ds\right].
	\end{multline}

	Taking the $\P$-expectation in the equation (\ref{eq_M2_dynamics}) and using eq. \eqref{eq_norm_estimate_integral} and Assumptio \ref{ass_non-degeneracy} we get
	\begin{align}\label{eq_estimate_M^2+H^2}
		0\leq &\E^{\P}\left[M_t^2+\int_t^T H_sH_s^T\colon d\qB_s\right]=\E^{\P}\left[M_T^2+2\int_t^TM_{s-}dK^c_t \right.\notag\\
			&\left. +2\int_t^TM_{s-}\sup_{v\in\v}\int_{\r0}K^d(s,z)v(dz)ds-\int_t^T\int_{\r0}\left[2M_{s-}K^d(s,z) +(K^d(s,z))^2\right]L(ds,dz)\right]\notag\\
				\leq& \E^{\P}\left[M_T^2+2\int_t^T|M_{s-}|dK^c_t+2\int_t^TM_{s-}\sup_{v\in\v}\int_{\r0}K^d(s,z)v(dz)ds \right]\notag\\
			& +\E^{\P}\left[2\int_t^T\sup_{v\in\v}\int_{\r0}|M_{s-}K^d(s,z)|v(dz)ds-\int_t^T\inf_{v\in\v}\int_{\r0}(K^d(s,z))^2v(dz)ds\right] \notag\\
			\leq& \E^{\P}\left[\sup_{s\in[t,T]}|M_s|^2+2\sup_{s\in[t,T]}|M_{s}|K^c_T+ \int_t^T\sup_{v\in\v}\int_{\r0}4|M_{s-}||K^d(s,z)|v(dz)ds \right]\notag\\
			&-\E^{\P}\left[\int_t^T\inf_{v\in\v}\int_{\r0}(K^d(s,z))^2v(dz)ds\right]\notag\displaybreak[3]\\
			\leq&(1+\epsilon^{-1}+4(T-t)\delta^{-1})\E^{\P} \left[\sup_{s\in[t,T]}|M_{s}|^2\right]+\epsilon\E^{\P}\left[(K^c_T)^2\right]\notag\\
			&+\E^{\P}\left[\delta\int_t^T\sup_{v\in\v}\int_{\r0}|K^d(s,z)|^2v(dz)ds-\int_t^T\inf_{v\in\v}\int_{\r0}(K^d(s,z))^2v(dz)ds\right]\notag\displaybreak[3]\\
			\leq&(1+\epsilon^{-1}+4(T-t)\delta^{-1})\E^{\P} \left[\sup_{s\in[t,T]}|M_{s}|^2\right]+\epsilon\E^{\P}\left[(K^c_T)^2\right]\notag\\
			&-(\underline{c}-\delta\bar{c})\E^{\P}\left[\int_t^T\int_{\r0}|K^d(s,z)|^2\pi(dz)ds\right]
	\end{align}
	where $\epsilon$ and $\delta$ are some positive constants. Will will use this equation three times. First, assume $\delta>0$ is small enough so that $\underline{c}-\delta\bar{c}>0$. Then it's trivial to get the estimate for $K^d$:
	\begin{equation}\label{eq_estimate_kd^2_1}
				\GE\left[\int_t^T\int_{\r0}|K^d(s,z)|^2\pi(dz)ds\right] \leq\frac{1+\epsilon^{-1}+4(T-t)\delta^{-1}}{\underline{c}-\delta\bar{c}}\GE \left[\sup_{s\in[t,T]}|M_{s}|^2\right]+\frac{\epsilon\GE\left[(K^c_T)^2\right]}{\underline{c}-\delta\bar{c}}
	\end{equation}
Then by the eq. (\ref{eq_representation_M}), continuity of \itolevy integral as an operator, H\"older inequality and again by eq. (\ref{eq_estimate_M^2+H^2}) and eq. (\ref{eq_estimate_kd^2_1}) we also get that
\begin{align}
	\E^{\P}\left[|K^c_T|^2\right] \leq& 5\E^{\P}\left[ |M_T|^2+|M_0|^2+\left(\int_0^T H_s\cdot dB_s\right)^2+\left(\int_0^T\int_{\r0} K^d(s,z)L(ds,dz)\right)^2\right.\notag\\
	&+\left.\left(\int_0^T \sup_{v\in\v} \int_{\r0} K^d(s,z)v(dz)ds\right)^2\right]\notag\\
	\leq& 5\E^{\P}\left[ \sup_{s\in[0,T]}|M_s|^2+|M_0|^2+\left(\int_0^T H_sH_s^T:d\qB_s\right)^2\right]\notag\\
		&+5\bar{c}(3T+1)\GE\left[\left(\int_0^T\int_{\r0} |K^d(s,z)|^2\pi(dz)ds\right)\right]\notag\\
	\leq& 5 \GE\left[\sup_{s\in[0,T]}|M_s|^2\right]+5(1+\epsilon^{-1}+
		4(T-t)\delta^{-1})\GE \left[\sup_{s\in[t,T]}|M_{s}|^2\right]+5\epsilon\GE\left[(K^c_T)^2\right]\notag\\
			&+5(\bar{c}(3T+1)-\underline{c}+\delta\bar{c})\GE\left[\int_0^T\int_{\r0}|K^d(s,z)|^2\pi(dz)ds\right]\notag\\
		\leq& 5 \GE\left[\sup_{s\in[0,T]}|M_s|^2\right]+5(1+\epsilon^{-1}+
		4T\delta^{-1})\GE \left[\sup_{s\in[0,T]}|M_{s}|^2\right]+5\epsilon\GE\left[(K^c_T)^2\right]\notag\\
			&+5\frac{\bar{c}(3T+1)-\underline{c}+\delta\bar{c}}{\underline{c}-\delta\bar{c}}\left[(1+\epsilon^{-1}+4T\delta^{-1})\GE \left[\sup_{s\in[0,T]}|M_{s}|^2\right]+\epsilon\GE\left[(K^c_T)^2\right]\right]\notag
\end{align}
Taking supremum over $\P\in\mathfrak{P}$ and rearranging this equation we get
\begin{align}
		\GE\left[|K^c_T|^2\right]\left[1-5\epsilon\frac{\bar{c}(3T+1)}{\underline{c}-\delta\bar{c}}\right]\leq 5\GE\left[\sup_{s\in[0,T]}|M_s|^2\right] \left[1+\left(1+\epsilon^{-1}+\frac{4T}{\delta}\right) \frac{\bar{c}(3T+1)}{\underline{c}-\delta\bar{c}}\right]
\end{align}
Fix $\delta:=\frac{\underline{c}}{2\bar{c}}$ and $\epsilon:=\frac{\underline{c}}{20\bar{c}(3T+1)}$. Then the coefficient on the LHS is equal to $1/2$ and we get the estimate
\begin{equation}\label{eq_estimate_K^c}
 	\GE\left[|K^c_T|^2\right]\leq 10\GE\left[\sup_{s\in[0,T]}|M_s|^2\right] \left[1+ 8\frac{\bar{c}^2(17T+5)(3T+1)}{\underline{c}^2}\right]=:C_1\,\GE\left[\sup_{s\in[0,T]}|M_s|^2\right]
\end{equation}
We can now use this estimate in eq. (\ref{eq_estimate_kd^2_1}) to get the existence of the constant $C_2$ such that
	\begin{equation}\label{eq_estimate_kd^2_2}
				\GE\left[\int_0^T\int_{\r0}|K^d(s,z)|^2\pi(dz)ds\right] \leq C_2\,\GE \left[\sup_{s\in[0,T]}|M_{s}|^2\right]
	\end{equation}
In the end we using this estimate and eq. (\ref{eq_estimate_K^c}) in eq. (\ref{eq_estimate_M^2+H^2}) we can get the existence of a constant $C_3$ such that 
	\begin{equation}\label{eq_estimate_H^2_2}
				\GE\left[\int_0^T H_sH_s^T\colon d\qB_s\right] \leq C_3\,\GE \left[\sup_{s\in[0,T]}|M_{s}|^2\right].
	\end{equation}
	Connecting these equations we get the assertion of the theorem.
\end{proof}

Using a very similar technique we may prove a theorem for the differences.

\begin{tw}\label{tw_estimates_differences}
	Let $M^i,\ i=1,2$ has the decomposition as below 
	\begin{equation}\label{eq_representation_Mbar}
	M^i_t = M^i_0 + \int_0^t H^i_s\cdot  dB_s -K^{i,c}_t +\int_0^t\int_{\r0} K^{i,d}(s,z)L(ds,dz)-\int_0^t \sup_{v\in\v} \int_{\R^d} K^{i,d}(s,z)v(dz)ds,\ q.s.,
\end{equation}
where $H^i\in\mathcal{H}^2_G(0,T;\R^d)$, $K^{i,c}$ is a non-decreasing process in $\s^2_G(0,T)$  such that $-K^{i,c}$ is a $G$-martingale and $K^{i,c}_0=0$ and $K^{i,d} \in \hgTR$, $i=1,2$.
Let $\bar{\cdot}$ denotes the difference between processes $\cdot^1$ and $\cdot^2$. Then under Assumption \ref{ass1} and \ref{ass_non-degeneracy} there exists a constant $C$ depending only on the dimension $d$ such that 
	\begin{align*}
		\|\bar{H}\|^2_{ \mathcal{H}^2_G(0,T;\R^d)}&+\|\bar{K}^c\|^2_{\s^2_G(0,T)}+\|\bar{K}^d\|^2_{\hgTR}\\
		&\leq C  \left[\|\bar{M}_{s}\|_{\s^2_G(0,T)}^2+\|\bar{M}_{s}\|_{\s^2_G(0,T)}\left(\|M^1_{s}\|_{\s^2_G(0,T)}+\|M^2_{s}\|_{\s^2_G(0,T)}\right)\right].
	\end{align*}
\end{tw}
\begin{proof}
	Just as in the proof of Theorem \ref{tw_estimates1} we use the \ito formula to get the following estimate
	\begin{align}
		0\leq \E^{\P}&\left[\bar{M}_t^2+\int_t^T \bar{H}_s\bar{H}_s^T\colon d\qB_s\right]=\E^{\P}\left[\bar{M}_T^2+2\int_t^T\bar{M}_{s-}d\bar{K}^c_t\right.\notag\\
		&+2\int_t^T\bar{M}_{s-}\left(\sup_{v\in\v}\int_{\r0}K^{1,d}(s,z)v(dz)ds-\sup_{v\in\v}\int_{\r0}K^{2,d}(s,z)v(dz) \right)ds\notag\\
			&\left. -\int_t^T\int_{\R^d}\left[2\bar{M}_{s-}\bar{K}^d(s,z) +(\bar{K}^d(s,z))^2\right]L(ds,dz)\right]\notag\\
				\leq& \E^{\P}\left[\bar{M}_T^2+2\int_t^T|M_{s-}|d(K^{1,c}_t+K^{2,c}_t)+2\int_t^T|\bar{M}_{s-}|\sup_{v\in\v}\int_{\r0}|\bar{K}^d(s,z)|v(dz)ds \right]\notag\\
			& +2\E^{\P}\left[\int_t^T\sup_{v\in\v}\int_{\r0}|\bar{M}_{s-}\bar{K}^d(s,z)|v(dz)ds\right] 
			-\E^{\P}\left[\int_t^T\inf_{v\in\v}\int_{\r0}(\bar{K}^d(s,z))^2v(dz)ds\right]\notag\displaybreak[3]\\
			\leq& \E^{\P}\left[\sup_{s\in[t,T]}|\bar{M}_s|^2+2\sup_{s\in[t,T]}|\bar{M}_{s}|(K^{1,c}_T+K^{2,c}_T)+ \int_t^T\sup_{v\in\v}\int_{\r0}4|\bar{M}_{s-}||\bar{K}^d(s,z)|v(dz)ds \right]\notag\\
			&-\E^{\P}\left[\inf_{v\in\v}\int_t^T\int_{{\r0}}(\bar{K}^d(s,z))^2v(dz)ds\right]\notag\displaybreak[3]\\
			\leq&(1+\frac{4(T-t)}{\delta})\E^{\P} \left[\sup_{s\in[t,T]}|\bar{M}_{s}|^2\right]+\E^{\P}\left[\sup_{s\in[t,T]}|\bar{M}_{s}|^2\right]^{1/2}\E^{\P}\left[(K^{1,c}_T+K^{2,c}_T)^2\right]^{1/2}\notag\\
			&+\E^{\P}\left[\delta\int_t^T\sup_{v\in\v}\int_{\r0}|\bar{K}^d(s,z)|^2v(dz)ds-\int_t^T\inf_{v\in\v}\int_{\r0}(\bar{K}^d(s,z))^2v(dz)ds\right]\notag\\
			\leq&(1+\frac{4(T-t)}{\delta})\E^{\P} \left[\sup_{s\in[t,T]}|\bar{M}_{s}|^2\right]+\E^{\P}\left[\sup_{s\in[t,T]}|\bar{M}_{s}|^2\right]^{1/2}\left[\E^{\P}[(K^{1,c}_T)^2]^{1/2}+\E^{\P}[(K^{2,c}_T)^2]^{1/2}\right]\notag\\
			&-(\underline{c}-\delta\bar{c})\E^{\P}\left[\int_t^T\int_{\r0}|\bar{K}^d(s,z)|^2\pi(dz)ds\right]
\end{align}		
Thus using the fact that we can estimate the norm of $K^{i,c}$ by the norm of $M^{i}$ by eq. (\ref{eq_estimate_K^c}) we can easily get that 
\begin{equation}\label{eq_estimate_H_diff}
\GE\left[\int_0^T \bar{H}_s\bar{H}_s^T\colon d\qB_s\right]\leq C_1 \left[\|\bar{M}_{s}\|_{\s^2_G(0,T)}^2+\|\bar{M}_{s}\|_{\s^2_G(0,T)}\left(\|M^1_{s}\|_{\s^2_G(0,T)}+\|M^2_{s}\|_{\s^2_G(0,T)}\right)\right]
\end{equation}
and 
\begin{equation}\label{eq_estimate_K_diff}
\GE\left[\int_0^T\int_{\r0}|\bar{K}^d(s,z)|^2\pi(dz)ds\right] \leq C_2 \left[\|\bar{M}_{s}\|_{\s^2_G(0,T)}^2+\|\bar{M}_{s}\|_{\s^2_G(0,T)}\left(\|M^1_{s}\|_{\s^2_G(0,T)}+\|M^2_{s}\|_{\s^2_G(0,T)}\right)\right].
\end{equation}

The estimate of the norm of $\bar{K}^c$ might be obtained analogously to the derivation of eq. (\ref{eq_estimate_K^c}) using the representation of $\bar{M}$.
\end{proof}

\section{Representation of $G$-martingales with a terminate value being a smooth cylinder random variables}

In this section we will use Proposition \ref{prop_existence_classical_phi} to prove that smooth cylinder random variables can be represented as the sum of the stochastic integral w.r.t. $G^c$-Brownian motion, a 'compensated' integral w.r.t. the pure-jump \levy process and a non-increasing continuous $G$-martingale. The procedure will be very similar to the one used in \cite{Peng_estimates} or \cite{Soner_mart_rep}, but we need to take into account the different structure of the operator $G$. 

Firstly, we will need the following easy lemma.
\begin{lem}\label{lem_convergence_in_LL_G}
	Let $(\xi_m)_{m\in\N}\subset Lip(\Omega_T)$ be a sequence of random variables of the form 
		\[
		\xi_m=\phi_m(X_{t_1},X_{t_2}-X_{t_1},\ldots,X_{t_n}-X_{t_{n-1}}).
	\]
	for some partition $0=t_0\leq t_1<t_2<\ldots<t_n\leq T$ and functions $\phi_m\in C^2(\R^{d\times n})$. Assume that $\phi_m$ converges uniformly to $\phi\in C^2(\R^{d\times n})$. Then $\xi_m$ converges to $\xi$ in $\L^2_G(\Omega_T)$, where $\xi$ is defined as
	\[
	\xi:=\phi(X_{t_1},X_{t_2}-X_{t_1},\ldots,X_{t_n}-X_{t_{n-1}}).
	\]
\end{lem}

\begin{proof}
	The assertion of the lemma follows quickly from the definition of the norm and the monotonicity of the (conditional) expectation:
	
	\begin{align*}
		\GE\left[\sup_{t\in[0,T]}\, \left(\GE[|\xi-\xi_m|\, |\F_t]\right)^2\right]&\leq \GE\left[\sup_{t\in[0,T]}\left(\GE[\sup_{x\in\R^{d\times n}}|\phi_m(x)-\phi(x)|\, |\F_t]\right)^2\right]\\&=[\sup_{x\in\R^{d\times n}}|\phi_m(x)-\phi(x)|]^2.\qedhere
	\end{align*}
\end{proof}

\begin{tw}\label{tw_representation_lip}
	Assume  Assumption \ref{ass1} and \ref{ass_non-degeneracy} and let $\xi\in \lipT$ be of the form
	\[
		\xi=\phi(X_{t_1},X_{t_2}-X_{t_1},\ldots,X_{t_n}-X_{t_{n-1}})
	\]
	for some partition $0=t_0\leq t_1<t_2<\ldots<t_n\leq T$ and a function $\phi\in C_b^2(\R^{d\times n})$. Then there exist unique processes $H\in \mathcal{H}^2_G(0,T;\R^d) $, $K^d\in\hgTR $ and $K^{c}\in \s^2_G(0,T)$, such that $K^c$ is a non-decreasing process m $-K^{c}$ is a $G$-martingale and
	\[
		\xi = \GE[\xi]+\int_0^T H_s\cdot  dB_s- K^c+\int_0^T\int_{\r0}K^d(s,z)L(ds,dz)-\int_0^T\sup_{v\in\v} \int_{\r0}K^d(s,z) v(dz)ds,\ q.s.
	\]
\end{tw}

\begin{proof}
Fix $\xi$ of the form
	\[
		\xi=\phi(X_{t_1},X_{t_2}-X_{t_1},\ldots,X_{t_n}-X_{t_{n-1}}).
	\]
Moreover, introduce the notation $x^k:=(x_1,\ldots,x_k)$ for an element of $\R^{d\times k}$ and $X^{k}:=(X_{t_1},X_{t_2}-X_{t_1},\ldots,X_{t_k}-X_{t_{k-1}})$, $k\leq n$.
 We will establish the representation backwards similarly to Theorem 15 in \cite{Peng_estimates}. Let $u$ be the classical solution of the equation
	\[
		\partial_t\, u^n_{x^{n-1}}(t,x)+G(D^2u^n_{x^{n-1}}(t,x),u^n_{x^{n-1}},t,x)=0,\quad u^n_{x^{n-1}}(t_n,x)=\phi({x^{n-1}},x).
	\]
	We put
	\[
		H^n_s:=\I_{]t_{n-1},t_n]}(s)Du^n_{X^{n-1}}(s,X_s-X_{t_{n-1}}),\]\[		K_s^{n,c}:=\int_{t_{n-1}}^sG^c(D^2u^n_{X^{n-1}}(s,X_s-X_{t_{n-1}}))ds-\frac{1}{2}\int_{t_{n-1}}^sD^2u^n_{X^{n-1}}(s,X_s-X_{t_{n-1}})\colon d\qB_s
	\]
	and
	\[
		K^{n,d}(s,z):=\I_{]t_{n-1},t_n]}(s)\left[u^n_{X^{n-1}}(s,X_{s-}-X_{t_{n-1}}+z)-u^n_{X^{n-1}}(s,X_{s-}-X_{t_{n-1}})\right]
	\]
	By the properties of $u^n_{x^{n-1}}$ it is clear that all processes belong to the appropriate spaces. Moreover, we can apply the \ito formula to obtain that
	\begin{align}\label{eq_mart_rep_lip_1}
		\xi-\GE[\xi|\F_{t_{n-1}}]&=u^n_{X^{n-1}}(t_{n},X_{t_n}-X_{t_{n-1}})-u^n_{X^{n-1}}(t_{n-1},0)\notag\\
		&=\int_{t_{n-1}}^{t_n} H^n_s\cdot dB_s-(K_{t_n}^{n,c}-K_{t_{n-1}}^{n,c})\notag\\&+\int_{t_{n-1}}^{t_n}\int_{\r0}K^{n,d}(s,z)L(ds,dz)-\int_{t_{n-1}}^{t_n}\sup_{v\in\v} \int_{\r0}K^{n,d}(s,z) v(dz)ds.
	\end{align}
	The natural reasoning would be to continue this procedure and try to solve the equation 
		\[
		\partial_t\, u^{n-1}_{x^{n-2}}(t,x)+G(D^2u^{n-1}_{x^{n-2}}(t,x),u^{n-1}_{x^{n-2}},t,x)=0
	\]
	with the terminal condition $u^{n-1}_{x^{n-2}}(t_{n-1},x)=\phi^{n-1}(x^{n-2},x):=u^n_{(x^{n-2},x)}(t_{n-1},0)$. Of course this problem has the solution $u^{n-1}_{x^{n-2}}$ in the viscosity sense, however the existence of the classical solution is more complicated due to the possible lack of smoothness of the terminal condition. To get rid of that problem, we will apply the approximation to the unity. Thus, let $\psi$ be a regular bump function on $\R^d$ and let
	\[
		\phi^{n-1}_{\epsilon}(x^{n-2},x):=\left(\phi^{n-1}(x^{n-2},\,.\,)\ast \psi_{\epsilon}\right)(x).
	\]
	Then, by the standard theory, $\phi^{n-1}_{\epsilon}(x^{n-2},\,.\,)$ is a smooth function converging uniformly on compact sets to $\phi^{n-1}(x^{n-2},\,.\,)$ as $\epsilon\downarrow0$ for each $x^{n-2}\in\R^{d\times(n-2)}$. However, due to the global Lipschitz continuity of $\phi^{n-1}$ we can get much more, namely the uniform (in $x\in \R^{d\times (n-1)}$) convergence of $\phi^{n-1}_{\epsilon}(.)\to \phi^{n-1}(.)$ as $\epsilon\downarrow0$. Hence, let $u^{n-1,\epsilon}_{x^{n-2}}$ denote the classical solution of the following integro-PDE.
	\[
		\partial_t\, u^{n-1,\epsilon}_{x^{n-2}}(t,x)+G(D^2u^{n-1,\epsilon}_{x^{n-2}}(t,x),u^{n-1,\epsilon}_{x^{n-2}},t,x)=0
	\]
	with the terminal condition $u^{n-1,\epsilon}_{x^{n-2}}(t_{n-1},x)=\phi^{n-1}_{\epsilon}(x^{n-2},x)$. Define processes
		\[
		H^{n-1,\epsilon}_s:=\I_{]t_{n-2},t_{n-1}]}(s)Du^{n-1,\epsilon}_{X^{n-2}}(s,X_s-X_{t_{n-2}}),\]
		\[		K_s^{n-1,\epsilon,c}:=\int_{t_{n-2}}^sG^c(D^2u^{n-1,\epsilon}_{X^{n-2}}(s,X_s-X_{t_{n-2}}))ds-\frac{1}{2}\int_{t_{n-2}}^sD^2u^{n-1,\epsilon}_{X^{n-2}}(s,X_s-X_{t_{n-2}})\colon d\qB_s
	\]
	and
	\[
		K^{n-1,\epsilon,d}(s,z):=\I_{]t_{n-2},t_{n-1}]}(s)\left[u^{n-1,\epsilon}_{X^{n-2}}(s,X_{s-}-X_{t_{n-2}}+z)-u^{n-1,\epsilon}_{X^{n-2}}(s,X_{s-}-X_{t_{n-2}})\right]
	\]
	Similarly to eq. (\ref{eq_mart_rep_lip_1}) we have
	\begin{align}\label{eq_mart_rep_lip_2}
		u^{n-1,\epsilon}_{X^{n-2}}&(t_{n-1},X_{t_{n-1}}-X_{t_{n-2}})-u^{n-1,\epsilon}_{X^{n-2}}(t_{n-2},0)	=\int_{t_{n-2}}^{t_{n-1}} H^{n-1,\epsilon}_s\cdot dB_s-(K_{t_{n-1}}^{n-1,\epsilon,c}-K_{t_{n-2}}^{n-1,\epsilon,c})\notag\\&+\int_{t_{n-2}}^{t_{n-1}}\int_{\r0}K^{n-1,\epsilon,d}(s,z)L(ds,dz)-\int_{t_{n-2}}^{t_{n-1}}\sup_{v\in\v} \int_{\r0}K^{n-1,\epsilon,d}(s,z) v(dz)ds.
	\end{align}
	Note also that $u^{n-1,\epsilon}_{X^{n-2}}(t_{n-1},X_{t_{n-1}}-X_{t_{n-2}})=\phi^{n-1}_{\epsilon}(X^{n-1})$ and $\GE[\xi\,|\,\F_{t_{n-1}}]=\phi^{n-1}(X^{n-1})$. Thus by the uniform convergence of $\phi^{n-1}_{\epsilon}$ to $\phi^{n-1}$ and Lemma \ref{lem_convergence_in_LL_G} we easily get that 
	\begin{equation}\label{eq_convergence_in_LL^2}
		u^{n-1,\epsilon}_{X^{n-2}}(t_{n-1},X_{t_{n-1}}-X_{t_{n-2}})\to \GE[\xi\,|\,\F_{t_{n-1}}],\quad \epsilon\downarrow 0\quad \textrm{in } \L^2_G(\Omega_T).
	\end{equation}

	Similarly, by the definition of the conditional expectation and its tower property we have that 
	\[
		\zeta^{\epsilon}:=u^{n-1,\epsilon}_{x^{n-2}}(t_{n-1},0)=\GE\left[u^{n-1,\epsilon}_{X^{n-2}}(t_{n-1},X_{t_{n-1}}-X_{t_{n-2}})\, |\F_{t_{n-2}}\right]
		=\GE\left[\phi^{n-1}_{\epsilon}(X^{n-1})\, |\F_{t_{n-2}}\right]
	\]
	and
	\[
		\zeta:=u^{n-1}_{x^{n-2}}(t_{n-1},0)=\GE[\xi\,|\,\F_{t_{n-2}}] =\GE[\phi^{n-1}(X^{n-1})\,|\,\F_{t_{n-2}}].
	\]
	Thus using the properties of the monotonicity and sublinearity of the conditional expectation, the tower property and the definition of the $\L^2_G$ norm we get
	\begin{align*}
		\GE\left[\sup_{t\in[0,T]}\, \left(\GE[|\zeta-\zeta^{\epsilon}|\, |\F_t]\right)^2\right]&\leq \GE\left[\sup_{t\in[0,T]}\left(\GE\left[\left|\GE\left[\phi^{n-1}_{\epsilon}(X^{n-1})-\phi^{n-1}(X^{n-1})\, |\F_{t_{n-2}}\right]\right|\, |\F_t\right]\right)^2\right]\\
		&\leq \GE\left[\sup_{t\in[0,T]}\left(\GE\left[\GE\left[\left|\phi^{n-1}_{\epsilon}(X^{n-1})-\phi^{n-1}(X^{n-1})\right|\,|\,\F_{t_{n-2}}\right]\, |\F_t\right]\right)^2\right]\displaybreak[3]\\
		&= \GE\left[\sup_{t\in[0,t_{n-2}]}\left(\GE\left[\left|\phi^{n-1}_{\epsilon}(X^{n-1})-\phi^{n-1}(X^{n-1})\right|\,|\,\F_{t}\right]\right)^2\right]\\
		&\leq \GE\left[\sup_{t\in[0,T]}\left(\GE\left[\left|\phi^{n-1}_{\epsilon}(X^{n-1})-\phi^{n-1}(X^{n-1})\right|\,|\,\F_{ t}\right]\right)^2\right].
	\end{align*}
	The last expression converges to $0$ because of the eq. (\ref{eq_convergence_in_LL^2}). Hence we get that $\zeta^{\epsilon}\to \zeta$ in $\L^2_G(\Omega_T)$ as $\epsilon\downarrow0$.  
	To sum up: we know that the LHS in eq. (\ref{eq_mart_rep_lip_2}) in $\L^2_G(\Omega_T)$ to a random variable $\GE[\xi\,|\,\F_{t_{n-1}}]-\GE[\xi\,|\,\F_{t_{n-2}}]$ as $\epsilon\downarrow0$ and by Theorem \ref{tw_estimates_differences} we get that the sequences $\{H^{n-1,\epsilon_m}\}_{m\in \N},\ \{K^{n-1,\epsilon_m,c}\}_{m\in \N},\ \{K^{n-1,\epsilon_m,d}\}_{m\in \N}$ are Cauchy sequences in appropriate spaces for any sequence $\{\epsilon_m\}_{m\in\N}$ such that $\epsilon_m\to 0$ as $m\to \infty$. Thus we also get the existence of the processes $H^{n-1},\ K^{n-1,c},\ K^{n-1,d}$ satisfying the following
		\begin{align*}
		\GE[\xi\,|\,\F_{t_{n-1}}]-&\GE[\xi\,|\,\F_{t_{n-2}}]	=\int_{t_{n-2}}^{t_{n-1}} H^{n-1}_s\cdot dB_s-(K_{t_{n-1}}^{n-1,c}-K_{t_{n-2}}^{n-1,c})\notag\\&+\int_{t_{n-2}}^{t_{n-1}}\int_{\r0}K^{n-1,d}(s,z)L(ds,dz)-\int_{t_{n-2}}^{t_{n-1}}\sup_{v\in\v} \int_{\r0}K^{n-1,d}(s,z) v(dz)ds.
	\end{align*}
	
	We can iterate this procedure to obtain processes $H^{i},\ K^{i,c},\ K^{i,d},\ i=1,\ldots,n$ such that for each $i\in\{1,\ldots,n\}$ we have similar representation
			\begin{align*}
		\GE[\xi\,|\,\F_{t_{i}}]-&\GE[\xi\,|\,\F_{t_{i-1}}]	=\int_{t_{i-1}}^{t_{i}} H^{i}_s\cdot dB_s-(K_{t_{i}}^{i,c}-K_{t_{i-1}}^{i,c})\notag\\&+\int_{t_{i-1}}^{t_{i}}\int_{\r0}K^{i,d}(s,z)L(ds,dz)-\int_{t_{i-1}}^{t_{i}}\sup_{v\in\v} \int_{\r0}K^{i,d}(s,z) v(dz)ds
	\end{align*}
	and thus putting $H:=\sum_{i=1}^n H^i,\ K^c:=\sum_{i=1}^n K^{i,c}$ and $K^d:=\sum_{i=1}^n K^{i,d}$ we obtain the desired representation for $\xi$.
\end{proof}

\section{Characterization of the random variables in the space $\L^2_G(\Omega)$.} 
In this section we show that $\L^2_G(\Omega_T)$ space is a large space and contains all random variables in $L^p_G(\Omega_T)$ space, $p>2$. Namely we have the following proposition.

\begin{prop}\label{prop_LL^2_G_contains_L^p_G}
Under Assumption \ref{ass1} for any $p>2$ there exists a constant $C_p$ such that for all $\xi\in Lip(\Omega_T)$ we have
\[
	\|\xi\|_{\L^2_G(\Omega)}\leq C_p \|\xi\|_{L^p_G(\Omega)}.
\]
\end{prop}

In a proof of the proposition we need the following lemma:
\begin{lem}\label{lem_doob}
	Assume Assumption \ref{ass1}. Let $\tau\leq T$ be stopping time. Fix $\P\in\p$, $\P^1,\ldots,\P^n\in\p(\tau,\P)$ and let $\{A_1,\ldots, A_n\}$ be a $\F_{\tau}$-measurable partition of $\Omega$. Then for any $\xi\in Lip(\Omega_T)$ we have
	\[
		\sum_{k=1}^n\E^{\P^k}[\xi\I_{A_k}]\leq \GE[\xi].
	\]
\end{lem}
\begin{proof}[Proof of Lemma \ref{lem_doob}]
	First introduce the following notation
	\[
		M^k_t:=\E^{\P^k}[\xi|\F_t],\quad \hat M_t:=\GE[\xi|\F_t].
	\]
	By the representation of the conditional sublinear expectation in Proposition \ref{prop_rep_cond_exp} it is easy to see that for all $t$ we have $M^k_t\leq \hat M_t\, \P^k-a.s.$ We need, however, a stronger property: that the $\P^k$-null set doesn't depend on $t$. To see that this property holds note that by the representation theorem we know that $\hat M$ has a q.s.-modification (hence also $\P^k$-modification) which has \cadlag paths apart from a polar set (which is also a $\P^k$-null set). Moreover, we can also choose the $\P$-modification of $M^k$ which has $\P^k$-a.a. \cadlag paths. The standard theorems for such regularity require the filtration to satisfy the usual conditions, whereas we work under the raw filtration which of course does not satisfy the usual conditions. But this is not a problem in our setting since  the measure $\P^k$ satisfies Blumenthal zero-one law as an push-forward measure of the law of a \levy process. Hence the augmented filtration $\{\F_t^{\P^k}\}_t$ is right-continuous. Moreover we can always choose a unique \cadlag modification of a \cadlag martingale $\E^{\P^k}[\xi|\F^{\P^k}_t]$ which is also a martingale w.r.t. unaugmented filtration $\fil=\{\F_t\}_t$ (compare with Lemma 2.1 in \cite{Soner_mart_rep} and Lemma 2.4 in \cite{Soner_quasi_anal}). Now taking $\P^k$-\cadlag modifications of $M^k$ and $\hat M$ we claim via standard arguments that
	\[
		\P^k(M^k_t\leq \hat M_t,\, \forall\, t\in[0,T])=1.	
	\]
As a consequence we have that
	\[
		1=\P^k(M^k_{\tau}\leq \hat M_{\tau})=\P(M^k_{\tau}\leq \hat M_{\tau})
	\]
	as $\P^k=\P$ on $\F_{\tau}$.
	
	Then we easily get via Doob's optional sampling theorem for a $\P^k$-martingale $M^k$ that
	\begin{align}
	\sum_{k=1}^n\E^{\P^k}[\xi\I_{A_k}]=\sum_{k=1}^n\E^{\P^k}[\I_{A_k}\E^{\P^k}[\xi|\F_{\tau}]]=\sum_{k=1}^n\E^{\P}[\I_{A_k}M^k_{\tau}]\leq\sum_{k=1}^n\E^{\P}[\I_{A_k}\hat M_{\tau}]=\E^{\P}[\hat M_{\tau}].
	\end{align}
	Note now that $\hat M$ is a $\P$-supermartingale (what is also an easy consequence of the representation of the conditional sublinear expectation), hence again by Doob's optional sampling theorem for $0$ and $\tau$ we get
		\[
	\sum_{k=1}^n\E^{\P^k}[\xi\I_{A_k}]\leq\E^{\P}[\hat M_{\tau}]\leq\E^{\P}[\hat M_0] =\hat M_0=\GE[\xi].\qedhere
	\]
\end{proof}
\begin{proof}[Proof of Proposition \ref{prop_LL^2_G_contains_L^p_G}]
	The proof follows the argument by \cite{Soner_mart_rep} in Lemma A.2. However we need to adjust a few details to take into consideration the fact that the measures constructed by Soner \emph{et al.} might not necessarily belong to the representation set $\mathfrak{P}$ in our setting. The adjustments however are minor.
	
	First note that without loss of generality we may take $\xi\geq0$. Let $M_t:=\GE[\xi|\F_t]$. $M_t$ has \cadlag paths q.s. by the representation theorem. By the representation of the conditional expectation we have that
	\[
		M_t=\sideset{}{^{\P}}\esup_{\Q\in\mathfrak{P}(t,\P)}\,\E^{\Q}[\xi|\F_t],\ \P-a.s.
	\]
	for every $\P\in\mathfrak{P}$. Define $M^*_t:=\sup_{0\leq s\leq t}\, M_s$. It suffices to show that 
	\[
		\E^{\P}[|M^*_T]^2]\leq C_p \|\xi\|_{L^p_G(\Omega)}\quad \textrm{for all}\quad \P\in\mathfrak{P}.
	\]
	We introduce $\tau_{\lambda}:=\inf\{t\geq0\colon M_t\geq \lambda\}$ for a fixed $\lambda>0$. $M$ is \cadlag and $\fil$-adapted, thus $\tl$ is a $\fil$-stopping time and we have
	\begin{equation}\label{eq_max_ineq_1}
			\P(\{M^*_T\geq\lambda\})=\P(\{\tl\geq T\})\leq \frac{1}{\lambda}\E^{\P}\left[M_{\tl}\I_{\{\tl\leq T\}}\right].
	\end{equation}
	Now there exists a sequence $\{\P_j,\ j\geq 1\}\subset \mathfrak{P}(\tl,\P)$ such that
	\[
		M_{\tl}=\sup_{j\geq1}\E^{\P_j}\left[\xi\, |\F_{\tl}\right]\,\ \P-a.s.
	\]
	For each $n\geq 1$ denote
		\[
		M^n_{\tl}=\sup_{1\leq j\leq n}\E^{\P_j}\left[\xi\, |\F_{\tl}\right].
	\]
	
	Introduce also sets 
	$\tilde{A}_j^{n}=\{M^n_{\tl}= \E^{\P_j}\left[\xi\, |\F_{\tl}\right]\},\ 1\leq j\leq n$, and put $A^{n}_1:=\tilde{A}^{n}_1$ and $A^{n}_j:=\tilde{A}^{n}_j\setminus \cup_{1\leq i \leq j-1} \tilde{A}_i^{n}$ for $2\leq j\leq n$. Then sets $\{A^{n}_j\}_j\subset \F_{\tl}$ form the partition of $\Omega$ for al $n\geq 1$. Hence we can introduce another probability measure $\hat\P^n$ by putting $\hat\P^n(A):=\sum_{j=1}^n\P^j(A\cap A^n_j)$. Note that $\hat \P^n=\P$ on $\F_{\tau}$.
	
	Fix $n\geq1$. Then we have for $q:=p/(p-1)$ by H\"older inequality and Lemma \ref{lem_doob}
	\begin{align*}
		\E^{\P}\left[M^n_{\tl}\I_{\{\tl\leq T\}}\right]=&\sum_{j=1}^n\E^{\P_j}\left[M^n_{\tl}\I_{\{\tl\leq T\}}\I_{A_j^n}\right]=\sum_{j=1}^n\E^{\P_j}\left[\E^{\P_j}\left[\xi\, |\F_{\tl}\right]\I_{\{\tl\leq T\}\cap A^n_j}\right]\\
		\leq&\sum_{j=1}^n\E^{\P_j}\left[\xi\,\I_{\{\tl\leq T\}\cap A_j^n}\right]=\E^{\hat\P^n}\left[\xi\,\I_{\{\tl\leq T\}}\right]\\
		 \leq &\left(\E^{\hat\P^n}\left[|\xi|^p\right]\right)^{\frac{1}{p}}\left(\hat\P^n\left(\{\tl\leq T\}\right)\right)^{\frac{1}{q}}\\
		 =&\left(\sum_{j=1}^n\E^{\P_j}\left[|\xi|^p\I_{A^n_j}\right]\right)^{\frac{1}{p}}\left(\P\left(\{\tl\leq T\}\right)\right)^{\frac{1}{q}}
		 \leq \left(\GE\left[|\xi|^p\right]\right)^{\frac{1}{p}}\left(\P(\{M^*_T\geq\lambda\})\right)^{\frac{1}{q}}.
	\end{align*}
	By going with $n$ to $\infty$ we get then that
	\[
	\E^{\P}\left[M_{\tl}\I_{\{\tl\leq T\}} \right]\leq \|\xi\|_{L^p_G(\Omega)}\left(\P(\{M^*_T\geq\lambda\})\right)^{\frac{1}{q}}.
	\]
	Plugging this estimate into eq. \eqref{eq_max_ineq_1} we get
	\[
	\P(\{M^*_T\geq\lambda\})\leq \frac{1}{\lambda}\E^{\P}\left[M_{\tl}\I_{\{\tl\leq T\}}\right]\leq  \frac{1}{\lambda} \|\xi\|_{L^p_G(\Omega)}\left(\P(\{M^*_T\geq\lambda\})\right)^{\frac{1}{q}}
	\]
	and hence
		\[
	\P(\{M^*_T\geq\lambda\})\leq  \frac{1}{\lambda^p} \|\xi\|^p_{L^p_G(\Omega)}.
	\]
	Just as Soner \emph{et al.} we fix $\lambda_0>0$ and compute
	\begin{align*}
		\E^{\P}[M^*_T]&=2\int_0^{\infty}\lambda\P(M^*_T>\lambda)d\lambda\leq 2\int_0^{\lambda_0}\lambda d\lambda +2\int_{\lambda_0}^{\infty}\lambda\P(M^*_T>\lambda)d\lambda\\
		&\leq \lambda_0^2+2 \|\xi\|^p_{L^p_G(\Omega)}\int_{\lambda_0}^{\infty}\frac{1}{\lambda^{p-1}}d\lambda= \lambda_0^2+2 \|\xi\|^p_{L^p_G(\Omega)}\frac{2}{p-2}\lambda_0^{2-p}.
			\end{align*}
		Hence taking $\lambda_0:= \|\xi\|_{L^p_G(\Omega)}$ we arrive at the inequality
		\[
			\E^{\P}[M^*_T]\leq C_p \|\xi\|^2_{L^p_G(\Omega)}
		\]
		and the constant $C_p$ doesn't depend on $\P$. To conclude we take the supremum over $\P\in\mathfrak{P}$.
\end{proof}

\end{document}